\documentclass[12pt]{amsart}
\usepackage[margin=1in]{geometry}
\usepackage{xcolor}
\usepackage[utf8]{inputenc}
\usepackage{amsfonts, amsthm, amssymb,amscd,amsmath, blindtext}
\usepackage{caption}
\usepackage{subcaption}
\usepackage{pdflscape}
\usepackage{multirow}
\usepackage{graphicx}
\usepackage{enumitem}
\usepackage{amssymb}
\usepackage{ytableau}
\usepackage{tikz}
\usepackage{tikz-cd}
\usepackage{float, pifont}
\usepackage{mathtools}
\usepackage[pagebackref,colorlinks=true,citecolor=blue,linkcolor=blue]{hyperref}
\usepackage{parskip}
\usepackage[capitalise, noabbrev]{cleveref}
\usepackage{fancyvrb}

\usepackage{algorithm}
\usepackage[noend]{algpseudocode}

\usetikzlibrary{arrows.meta}

\newtheorem{theorem}{Theorem}[section]
\newtheorem{lemma}[theorem]{Lemma}
\newtheorem{corollary}[theorem]{Corollary}
\newtheorem{proposition}[theorem]{Proposition}

\newtheorem{claim}[theorem]{Claim}

\theoremstyle{remark}
\newtheorem{remark}[theorem]{Remark}
\newtheorem{question}[theorem]{Question}

\theoremstyle{definition}
\newtheorem{definition}[theorem]{Definition}
\newtheorem{ex}[theorem]{Example}

\newtheorem{construction}{Construction}
\newtheorem{notation}[theorem]{Notation}
\newtheorem{setup}[theorem]{Set-Up}

\DeclareMathOperator{\Tor}{Tor}
\DeclareMathOperator{\reg}{reg}

\DeclareMathOperator{\lcm}{lcm}

\DeclareMathOperator{\depth}{depth}
\DeclareMathOperator{\pd}{pdim}
\DeclareMathOperator{\supp}{supp}

\newcommand{\R}{\mathbb{R}}

\newcommand{\F}{\mathbb{F}}

\newcommand{\cw}{\mathbf{c}}
\newcommand{\C}{\mathcal{C}}

\newcommand{\U}{\mathcal{U}}

\def\P{{\mathcal P}}

\def\p{{\bf p}}

\def\x{{\bf x}}
\def\y{{\bf y}}

\def\1{{\bf 1}}
\def\0{{\bf 0}}
\def\r{{\mathbf r}^{\mathbb L}}

\begin{document}

\title{Neural codes via homological invariants of polarized neural ideals}

\author{Selvi Kara}
\address{Department of Mathematics, Bryn Mawr College,  906 New Gulph Rd, Bryn Mawr, PA 19010, USA}
\email{skara@brynmawr.edu}

\author{Ellie Lew}
\address{Department of Mathematics, University of California, One Shields Avenue, Davis, CA 95616-8633, USA}
\email{ejlew@ucdavis.edu}

\begin{abstract}
For a neural code $\C\subseteq\F_2^n$, polarizing the canonical form generators of the neural ideal $J_\C$
yields a squarefree monomial ideal $\P(J_\C)\subset k[x_1,\dots,x_n,y_1,\dots,y_n]$, the \emph{polarized neural ideal}, and an associated
simplicial complex $\Delta_\C$, the \emph{polar complex}.  We study the graded invariants
$\pd(\P(J_\C))$ and $\reg(\P(J_\C))$ via the topology of $\Delta_\C$, showing that simple geometric
features of the Hamming cube  organize their extremal behavior. 
We prove $\reg(\P(J_\C))\le 2n-1$, with equality precisely when $\C$ is obtained from $\F_2^n$ by deleting
an antipodal pair. Using connectedness properties of induced subcomplexes of
$\Delta_\C$, we obtain $\pd(\P(J_\C))\le 2n-3$, and we give an explicit family of codes attaining equality, each consisting of antipodal pairs.
At the opposite end, we identify the cube geometry behind the smallest values: $\reg(\P(J_\C))=1$ forces $\C$
to be a coordinate subcube of $\F_2^n$, while $\pd(\P(J_\C))=0$ forces $\C$ to be the complement of one.
Finally, we construct families realizing large regions of the $(\pd,\reg)$-plot for fixed $n$.
\end{abstract}

\maketitle

\section{Introduction}\label{sec:intro}
Combinatorial neural codes provide a discrete model for patterns of neural activity.  A code
$\C\subseteq \F_2^n$ records which subsets of $n$ neurons can be simultaneously active across stimuli.
Beyond being a list of observed firing patterns, a code implicitly encodes constraints among receptive
fields and, more broadly, structural information about an underlying stimulus representation that is
shared across realizations.  Extracting such realization invariant structure is subtle, and algebra
provides a systematic way to do so.  In particular, Curto et al.~\cite{smb}
introduced \emph{neural ideals} as an algebraic encoding of essential receptive-field relations; see also
\cite{acc,rgg, youngsphd} and the references therein.

To a code $\C\subseteq \F_2^n$ one associates a neural ideal $J_\C\subset k[x_1,\dots,x_n]$.  The ideal
$J_\C$ can be studied via its \emph{canonical form} $CF(J_\C)$, a distinguished generating set
consisting of the minimal pseudo-monomials encoding minimal receptive-field relations; this set
generates $J_\C$ (not necessarily minimally).  Since neural ideals are typically not graded, we pass
to a graded setting by polarizing the canonical form generators, following~\cite{gjs}: replacing each
factor $(1-x_i)$ by a new variable $y_i$ produces a squarefree monomial ideal
\[
\P(J_\C)\ \subset k[x_1,\dots,x_n,y_1,\dots,y_n],
\]
the \emph{polarized neural ideal} of $\C$.  This is the main algebraic object we study in this paper.
Our goal is to organize neural codes using homological data of $\P(J_\C)$.  We focus on the graded
invariants $\pd(\P(J_\C))$ and $\reg(\P(J_\C))$, namely the projective dimension and
Castelnuovo--Mumford regularity of $\P(J_\C)$.

Because $\P(J_\C)$ is squarefree, it is the Stanley--Reisner ideal of a simplicial complex $\Delta_\C$,
which we call the \emph{polar complex} (following the terminology of~\cite{gjs}).  In particular, the
multigraded Betti numbers of $\P(J_\C)$, and hence $\pd(\P(J_\C))$ and $\reg(\P(J_\C))$, are controlled
by the topology of induced subcomplexes of $\Delta_\C$ via Hochster's formula.  This
correspondence forms the foundation of our homological study of neural codes.

Several authors have associated simplicial complexes to neural codes in order to study geometric and
combinatorial properties complementary to the homological questions pursued here.  For example,
\cite{IKR_hyperplane_polar} introduces a polar complex $\Gamma(\C)$  and uses
shellability to analyze stable hyperplane codes, while \cite{PMS_factor_complex} defines the factor
complex $\Delta_\cap(\C)$ to study interval structure and (max-)intersection-complete codes.  Neither
$\Gamma(\C)$ nor $\Delta_\cap(\C)$ agrees in general with the polar complex $\Delta_\C$ considered in
this paper.

Although $\P(J_\C)$ is a squarefree monomial ideal in $2n$ variables, it is far from arbitrary: it
satisfies the ambient constraint that no minimal generator is divisible by $x_i y_i$.  Homological
properties of squarefree monomial ideals satisfying this condition have been examined recently
in~\cite{chau2025neural}.  Our approach differs from~\cite{chau2025neural} in that we repeatedly
exploit additional structure coming from the fact that $\P(J_\C)$ arises by polarizing canonical
form generators, a hypothesis not imposed there.

A guiding theme of this work is that extremal and boundary behavior for $\pd(\P(J_\C))$ and
$\reg(\P(J_\C))$ is controlled by elementary combinatorial features of the \emph{Hamming cube} in which
the code $\C$ lives. Throughout, by the \emph{Hamming cube} we mean the $n$--dimensional hypercube
with vertex set $\F_2^n=\{0,1\}^n$, equipped with the Hamming distance
\[
d(u,v)=|\{i\in[n]:u_i\neq v_i\}|,
\]
so that two vertices are adjacent precisely when they differ in exactly one coordinate.

\begin{figure}[ht]
  \centering
  \scalebox{0.85}{
  \begin{minipage}{0.48\linewidth}
    \centering
    \begin{tikzpicture}[x=0.7cm,y=0.7cm]
    \draw[step=1,gray!65,line width=.35pt] (0,0) grid (6,7);
    \draw[very thick,-{Stealth[length=6pt]}] (0,0) -- (6.55,0) node[below right] {$\pd$};
    \draw[very thick,-{Stealth[length=6pt]}] (0,0) -- (0,7.55) node[above left] {$\reg$};

      \foreach \x in {0,1,2,3,4,5} { \node[below] at (\x,0) {\x}; }
      \foreach \y in {0,1,2,3,4,5,6,7} { \node[left]  at (-0.2,\y) {\y}; }

      \tikzset{dot/.style={circle, fill=black, inner sep=0pt, minimum size=8pt}}
      \foreach \p/\r in {
        0/1,0/2,0/3,
        1/1,1/2,1/3,1/4,1/5,
        2/1,2/3,2/4,
        3/3
      }{
        \node[dot] at (\p,\r) {};
      }
    \end{tikzpicture}
    \subcaption{$n=3$}
    \label{fig:n=3}
  \end{minipage}
  \hfill
 \begin{minipage}{0.48\linewidth}
    \centering
\begin{tikzpicture}[x=0.7cm,y=0.7cm]
  \draw[step=1,gray!65,line width=.35pt] (0,0) grid (6,7);
  \draw[very thick,-{Stealth[length=6pt]}] (0,0) -- (6.55,0) node[below right] {$\pd$};
  \draw[very thick,-{Stealth[length=6pt]}] (0,0) -- (0,7.55) node[above left] {$\reg$};

  \foreach \x in {0,...,5} \node[below] at (\x,0) {\x};
  \foreach \y in {0,...,7} \node[left]  at (-0.2,\y) {\y};

  \tikzset{dot/.style={circle,fill=black,inner sep=0pt,minimum size=8pt}}
  \foreach \p/\r in {
    5/3,
    4/3,4/4,4/5,
    3/1,3/3,3/4,3/5,3/6,
    2/1,2/2,2/3,2/4,2/5,2/6,
    1/1,1/2,1/3,1/4,1/5,1/6,1/7,
    0/1,0/2,0/3,0/4
  } \node[dot] at (\p,\r) {};
\end{tikzpicture}
    \subcaption{$n=4$}
    \label{fig:n=4}
  \end{minipage}
  }
  \caption{Achievable $(\pd,\reg)$ pairs for polarized neural ideals when $n=3,4$.}
  \label{fig:plots_3_4}
\end{figure}
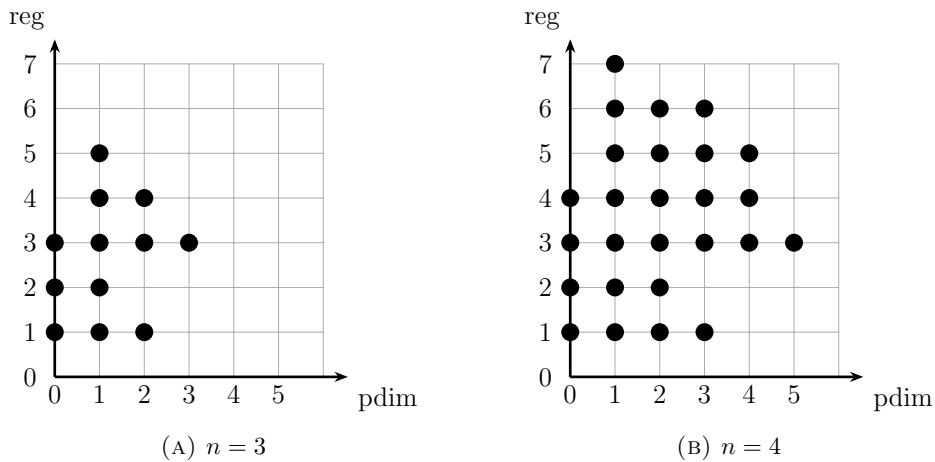

Our first results provide general constraints from the topology of $\Delta_\C$.
We prove that for every nonempty code $\C$, the polar complex $\Delta_\C$ is connected, and more
generally that every induced subcomplex on at least $n+1$ vertices is connected
(\Cref{lem:DW-connected-largeW,cor:polar-complex-connected}).  This threshold yields a sharp bound
$\pd(\P(J_\C))\le 2n-3$ for $n\geq 2$, and we exhibit broad families achieving equality
(\Cref{thm:pdim_upper_bound,thm:all-or-nothing,cor:antipodal_pairs_maxpdim}).  We also establish the
upper bound $\reg(\P(J_\C))\le 2n-1$ and classify the extremal case
\[
\reg(\P(J_\C))=2n-1 \quad\Longleftrightarrow\quad \C=\F_2^n\setminus\{v,\bar v\}
\]
(\Cref{lem:reg_upper_bound,thm:maxreg_classification}), as well as the opposite extreme $\reg=1$, which
occurs precisely for coordinate subcubes (\Cref{thm:reg1-classification}).

A recurring phenomenon is that Hamming-cube geometry controls both invariants at their extremes.
The smallest nonzero regularity, $\reg(\P(J_\C))=1$, occurs precisely when $\C$ is a coordinate subcube
(equivalently $\C=Q(\tau,\sigma)$ in the notation of \Cref{prop:pdim0_classification}).
By contrast, the smallest projective dimension, $\pd(\P(J_\C))=0$, occurs precisely when $\C$ is the complement
of a nonempty coordinate subcube, i.e.\ $\C=\F_2^n\setminus Q(\sigma,\tau)$
(\Cref{thm:reg1-classification,prop:pdim0_classification}). Moreover, these two boundary families are related by Alexander duality:  $\reg(\P(J_\C))=1$ iff $\P(J_\C)$ is generated by variables, while $\pd(\P(J_\C))=0$ iff $\P(J_\C)$ is principal, and Alexander duality swaps these two classes.
At the
other end, maximal regularity is forced by deleting an antipodal pair, whereas maximal projective
dimension is already attained by any antipodal pair code $\{v,\bar v\}$. We exploit this interplay
between fixing coordinates, taking complements of coordinate subcubes, and antipodality throughout the
paper, and these results support the broader theme that the homological invariants of $\P(J_\C)$
provide a useful organizing lens for neural codes.

Beyond extremal classifications, we give explicit constructions realizing large regions of the
$(\pd,\reg)$-plot for fixed $n$.  We develop operations (free-neuron and constant-neuron extensions)
that change $n$ while controlling $\pd$ and $\reg$ (\Cref{lem:free-neuron,lem:constant-neuron}).  We use
combinatorial patterns of missing codewords in the Hamming cube to control the topology
of $\Delta_\C$, hence the Betti numbers via Hochster’s formula.  In particular, we realize the full
region $\pd(\P(J_\C))+\reg(\P(J_\C))\le n$ (\Cref{thm:pr-band-realization}), and we obtain complete
classification/realizability results along the low lines $\pd=0,1$ and $\reg=1,2,3$
(\Cref{sec:small-pdim,sec:small-reg}).

\textbf{Organization.}
\Cref{sect:neural-codes-intro} reviews neural codes, neural ideals, canonical forms, polarization, and
the polar complex.  \Cref{sec:toolkit} records code operations controlling $\pd$ and $\reg$ and \Cref{subsec:polar-connectedness} proves
connectedness properties of $\Delta_\C$.  \Cref{sect:results} establishes the sharp upper bound
$\pd\le 2n-3$, and \Cref{sec:regularity} studies regularity, including the maximal-regularity
classification.  \Cref{sec:small-pdim,sec:small-reg} treat the low lines $\pd=0,1$ and $\reg=1,2,3$,
and \Cref{sec:last} gives the simplicial-sphere construction realizing $\pd+\reg\le n$.
\Cref{sec:further-directions} concludes with open directions.

\section{Preliminaries}\label{sect:neural-codes-intro}

This section collects background and notation used throughout the paper. We follow the algebraic
framework of~\cite{smb}, which introduces the neural ideal and its canonical form as a minimal
algebraic encoding of receptive-field (RF) relations. We then recall polarization of neural ideals
and the associated \emph{polar complex} from ~\cite{gjs}.

\begin{setup}\label{setup:prelim}
Fix a field $k$. For $n\ge 1$ set $R:=k[x_1,\dots,x_n]$ and
$S:=k[x_1,\dots,x_n,y_1,\dots,y_n]$, both standard $\Bbb Z$-graded with
$\deg(x_i)=\deg(y_i)=1$. We write $[n]=\{1,\dots,n\}$ and identify $\F_2^n=\{0,1\}^n\subset k^n$.
\end{setup}

\subsection{Neural codes and realizations}

Many of the codes we study arise from a family of receptive fields $\U=\{U_1,\dots,U_n\}$ in a
stimulus space $X$.  The resulting code records which firing patterns occur and, equivalently, which
set-theoretic containments among intersections $\bigcap_{i\in\sigma}U_i$ and unions $\bigcup_{j\in\tau}U_j$
are forced by the absence of certain patterns.

\begin{definition}
A \emph{neural code} on $n$ neurons is a subset $\C\subseteq \F_2^n$ whose elements
$\cw=(\cw_1,\dots,\cw_n)$ are called \emph{codewords}. The $i$th coordinate $\cw_i$ indicates
whether neuron $i$ fires ($\cw_i=1$) or does not fire ($\cw_i=0$).
\end{definition}

\begin{definition}\label{def:code_of_realization}
Let $X\neq\emptyset$ and $\U=\{U_1,\dots,U_n\}$ with $U_i\subseteq X$.
The \emph{code of $\U$} is
\[
\C(\U) := \left\{\cw\in\F_2^n: 
\left(\bigcap_{\cw_i=1}U_i\right)\setminus\left(\bigcup_{\cw_j=0}U_j\right)\neq\emptyset
\right\}.
\]
If $\C=\C(\U)$ for some $\U$, we say $\U$ is a \emph{realization} of $\C$.
\end{definition}

\begin{figure}[ht]
    \centering
    \begin{subfigure}[b]{0.5\textwidth}
        \centering
            \includegraphics[width = 5cm]{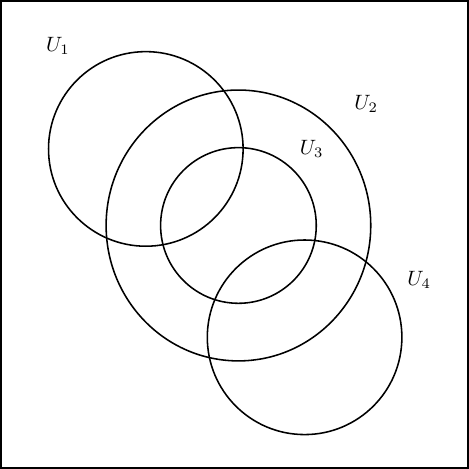}
        \caption{Collection $\U = \{U_1,U_2,U_3,U_4\}$}
        \label{fig-1-a}
    \end{subfigure}%
    ~
    \begin{subfigure}[b]{0.5\textwidth}
        \centering
        \includegraphics[width = 5cm]{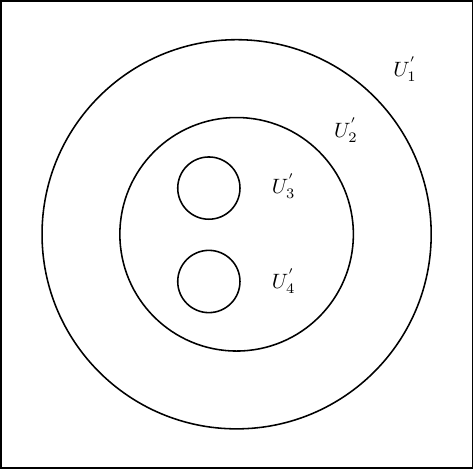}
        \caption{Collection $\U' = \{U_1',U_2',U_3',U_4'\}$}
        \label{fig-1-b}
    \end{subfigure}
    \caption{Two collections of receptive fields in $\R^2$.}
    \label{fig:ex1}
\end{figure}

\begin{ex}
  For the collections $\U$ and $\U'$ in \Cref{fig:ex1}, the associated codes are
\begin{align*}
\C(\U) &=\{0000,1000,0100,0001,1100,1110,0101,0111\}, \\
\C(\U')&=\{0000,1000,1100,1101,1110\}.
\end{align*}  
\end{ex}

Although a code $\C$ can have many realizations, the code determines a collection of RF relations
that must hold in \emph{every} realization.  A basic family of such relations are containments of the form
\[
\bigcap_{i\in\sigma}U_i \subseteq \bigcup_{j\in\tau}U_j,
\qquad \sigma,\tau\subseteq[n].
\]

\subsection{Neural ideals and RF relations}\label{sect:neural-ideal}

Neural ideals provide an algebraic way to record RF containments. We work in the polynomial ring
$R=k[x_1,\dots,x_n]$.

\begin{definition}\label{def:neural_ideal}
Let $\C\subseteq\F_2^n$ be a neural code. For $v\in\F_2^n$ define its \emph{characteristic
pseudo-monomial}
\[
\rho_v := \prod_{v_i=1}x_i\prod_{v_j=0}(1-x_j)\in R.
\]
The \emph{neural ideal} of $\C$ is
\[
J_\C := \langle \rho_v : v\in\F_2^n\setminus \C\rangle \subseteq R.
\]
\end{definition}

\begin{remark}\label{rem:boolean-zeroset}
For each $v\in\F_2^n$, the pseudo-monomial $\rho_v$ satisfies $\rho_v(v)=1$ and $\rho_v(w)=0$
for all $w\in\F_2^n$ with $w\neq v$. Consequently, the zero set of $J_\C$ on $\F_2^n$ is
\[
V(J_\C)=\{w\in\F_2^n:\ f(w)=0\ \text{for all}\ f\in J_\C\}=\C.
\]
See also~\cite[Lemma~3.2]{smb}.
\end{remark}

\begin{notation}
For $\sigma\subseteq[n]$, write
\[
U_{\sigma}:=\bigcap_{i\in \sigma} U_i
\qquad \text{and} \qquad
x_\sigma:=\prod_{i\in\sigma}x_i.
\]
\end{notation}

The following theorem explains how pseudo-monomials encode RF containments.

\begin{theorem}\label{thm:rf-neural-ideal}\cite{smb}
Let $\C\subseteq\F_2^n$ and let $\U=\{U_1,\dots,U_n\}$ be a realization in a stimulus space $X$.
Fix $\sigma,\tau\subseteq[n]$ with $\sigma\cap\tau=\emptyset$. Then
\[
x_\sigma\prod_{j\in\tau}(1-x_j)\in J_\C
\iff
\bigcap_{i\in\sigma}U_i \subseteq \bigcup_{j\in\tau}U_j.
\]
\end{theorem}

\begin{notation}
We use the conventions
\[
\bigcap_{i\in\emptyset}U_i=X
\qquad\text{and}\qquad
\bigcup_{j\in\emptyset}U_j=\emptyset.
\]
\end{notation}

\subsection{Canonical form, and Type~1--3 relations}\label{sect:cf}

Neural ideals are generated by many characteristic pseudo-monomials, and it is useful to extract a
generating set capturing the essential RF relations. In~\cite{smb}, Curto et al.\ introduced the
\emph{canonical form} of a neural ideal and organized its relations into three types.

\begin{definition}\label{def:canonical-form}
Let $I\subseteq R$ be an ideal.
\begin{itemize}
\item A \emph{pseudo-monomial} is a polynomial of the form
\[
x_\sigma\prod_{j\in\tau}(1-x_j),
\qquad \sigma,\tau\subseteq[n],\ \sigma\cap\tau=\emptyset.
\]
\item A pseudo-monomial $f\in I$ is \emph{minimal} if no proper divisor of $f$ is contained in $I$.
\item The \emph{canonical form} $CF(I)$ is the set of all minimal pseudo-monomials in $I$.
\end{itemize}
\end{definition}

\begin{theorem}[{\cite[Theorem~4.3]{smb}}]\label{thm:curto_CF}
For any neural code $\C\subseteq\F_2^n$, the neural ideal $J_\C$ is generated by its canonical form:
\[
J_\C=\langle CF(J_\C)\rangle.
\]
Moreover, $CF(J_\C)$ decomposes as a disjoint union of minimal relations of three types:
\begin{itemize}
\item \textbf{Type 1:} monomials $x_\sigma$, encoding $U_\sigma=\emptyset$;
\item \textbf{Type 2:} pseudo-monomials $x_\sigma\prod_{i\in\tau}(1-x_i)$, encoding
$U_\sigma\subseteq \bigcup_{i\in\tau}U_i$;
\item \textbf{Type 3:} pseudo-monomials $\prod_{i\in\tau}(1-x_i)$, encoding
$X\subseteq \bigcup_{i\in\tau}U_i$.
\end{itemize}
% In each case, minimality in $CF(J_\C)$ means that the associated containment fails after replacing
% $\sigma$ and/or $\tau$ by a proper subset.
\end{theorem}

\begin{remark}\label{rem:minimal_rf}
If $f=x_\sigma\prod_{j\in\tau}(1-x_j)\in CF(J_\C)$, then the corresponding containment
$U_\sigma\subseteq \bigcup_{j\in\tau}U_j$ holds, but it fails if $\sigma$ or $\tau$ is replaced by a
proper subset. Equivalently, $f$ is minimal among pseudo-monomials in $J_\C$ under divisibility.
See \cite[Theorem~4.3]{smb}.
\end{remark}

\begin{remark}\label{rem:CF-not-minimal}
The canonical form $CF(J_\C)$ is a canonical generating set by pseudo-monomials: its elements
are minimal under divisibility among pseudo-monomials in $J_\C$, and $J_\C=\langle CF(J_\C)\rangle$.
However, this does \emph{not} imply that $CF(J_\C)$ is a minimal generating set of $J_\C$ as an ideal
in the usual sense, since $J_\C$ may admit a generating set with fewer elements
once arbitrary polynomials (not necessarily pseudo-monomials) are allowed.
\end{remark}

\subsection{Polarization and the polar complex}\label{sect:neural-polarization}

Neural ideals need not be homogeneous.
To access graded invariants (graded Betti numbers, regularity, projective dimension), one replaces
each factor $(1-x_i)$ by a new variable $y_i$, producing a squarefree monomial ideal.

\begin{definition}[{\cite{gjs}}]\label{def:polarization}
Let $f=x_\sigma\prod_{j\in\tau}(1-x_j)$ be a pseudo-monomial with $\sigma\cap\tau=\emptyset$.
Its \emph{polarization} is the squarefree monomial
\[
\P(f) := x_\sigma y_\tau \in S.
\]
Then  the \emph{polarized neural ideal} is
\[
\P(J_\C) := \langle \P(f) :  f\in J_\C \text{ is a pseudo-monomial }\rangle \subseteq S.
\]
\end{definition}

\begin{theorem}\cite[Theorem 3.2]{gjs}
Let $J_\C=\langle f_1,\dots,f_k\rangle$ be a neural ideal  in canonical form, i.e.\
$\{f_1,\dots,f_k\}=CF(J_\C)$. Then 
\[
\P(J_\C)=\langle \P(f_1),\dots,\P(f_k)\rangle \subseteq S.
\]
\end{theorem}

We say that the polarized neural ideal $\P(J_\C)$ is in canonical form if it is written as
\[
\P(J_\C)=\langle \P(f_1),\dots,\P(f_k)\rangle. 
\]

\begin{remark}
It is shown in~\cite[Theorem~3.5 and Corollary~4.1]{gjs} that polarization preserves the algebraic
structure of the neural ideal in a precise sense: one can recover a free resolution of $J_\C$ from a
minimal free resolution of $\P(J_\C)$ by substituting $y_i\mapsto (1-x_i)$ in the presentation matrices.
\end{remark}

Since $\P(J_\C)$ is a squarefree monomial ideal, it corresponds to a simplicial complex via
Stanley--Reisner theory.

\begin{notation}\label{not:polar-complex}
Let $\C\subseteq\F_2^n$. We write $\Delta_\C$ for the simplicial complex whose Stanley--Reisner ideal is $\P(J_\C)\subseteq S$.
We call $\Delta_\C$ the \emph{polar complex} of $\C$ as in ~\cite{gjs}.
\end{notation}

\begin{definition}\label{def:transversal_monomial}
A squarefree monomial $m\in S$ is \emph{transversal} if $x_i y_i\nmid m$ for all $i\in[n]$.
Equivalently, for each $i$, the monomial $m$ contains at most one of $x_i$ and $y_i$.

For $v\in\F_2^n$ with $\sigma(v)=\{i\in[n]: v_i=1\}$, define the associated transversal monomial
\[
m_v  :=  x_{\sigma(v)} y_{[n]\setminus\sigma(v)}
 = \prod_{i\in\sigma(v)} x_i \prod_{i\in[n]\setminus\sigma(v)} y_i.
\]
\end{definition}

\subsection{Graded Betti numbers, projective dimension, and regularity}\label{subsec:betti-pd-reg}

Let $S$ be a standard $\Bbb Z$-graded polynomial ring,
and $M$ be a finitely generated $\Bbb Z$-graded $S$-module. A \emph{minimal graded free resolution}
of $M$ has the form
\[
0 \longrightarrow \bigoplus_j S(-j)^{\beta_{\ell,j}(M)}
\longrightarrow \cdots \longrightarrow
\bigoplus_j S(-j)^{\beta_{0,j}(M)}
\longrightarrow M \longrightarrow 0,
\]
where  the integers $\beta_{i,j}(M)$ are the \emph{graded Betti numbers}.
Equivalently,
\[
\beta_{i,j}(M)=\dim_k \Tor_i^S(M,k)_j .
\]

The \emph{projective dimension} and \emph{Castelnuovo--Mumford regularity} of $M$ are defined by
\[
\pd(M):=\max\{  i : \beta_{i,j}(M)\neq 0 \},
\qquad
\reg(M):=\max\{  j-i : \beta_{i,j}(M)\neq 0 \}.
\]
Throughout the paper, if $I\subseteq S$ is a homogeneous ideal, we regard $I$ as a graded $S$-module
and apply these definitions to $M=I$.

In this paper, we compute graded Betti numbers of $\P(J_\C)$ using Hochster’s formula.

\begin{theorem}[Hochster's formula]\label{thm:hochster}
Let $\Delta$ be a simplicial complex on the vertex set $V=\{x_1,\dots,x_m\}$ with
Stanley--Reisner ideal $I_\Delta\subseteq k[x_1,\dots,x_m]$. Then for all $i,j$,
\[
\beta_{i,j}(I_\Delta)
=\sum_{\substack{W\subseteq V\\|W|=j}}
\dim_k \widetilde H_{j-i-2}(\Delta_W;k),
\]
where $\Delta_W$ is the induced subcomplex on $W$ and $\widetilde H_\ast$ denotes reduced simplicial
homology.
\end{theorem}

We  repeatedly use the following standard fact about adjoining new variables.

\begin{lemma}\label{lem:add_variable_pd_reg}
Let $M$ be a finitely generated graded $S$-module,
and let $z$ be a new variable of degree $1$. Set $S':=S[z]$ and $M':=M\otimes_S S'$.
Then
\[
\pd_{S'}(M')=\pd_S(M)
\qquad\text{and}\qquad
\reg_{S'}(M')=\reg_S(M).
\]
Moreover, for any proper homogeneous ideal $J\subsetneq S$, writing $J S'\subseteq S'$ for its extension,
one has
\[
\pd_{S'}(J S'+(z))=\pd_{S}(J)+1,
\qquad\text{and}\qquad
\reg_{S'}(J S'+(z))=\reg_{S}(J).
\]
\end{lemma}

\section{Basic operations on codes}\label{sec:toolkit}

This section records conventions and two elementary operations on neural codes that  will be used
throughout the paper.  The first adjoins a \emph{free neuron} (a new coordinate that can be chosen
independently), and the second adjoins a \emph{constant neuron} (a new coordinate fixed to $0$ or $1$).
On the algebraic side, these correspond to extension by new variables, and (in the constant
case) adding a new linear generator.  In particular, these operations allow us to vary the number of
neurons while controlling $(\pd,\reg)$.

\begin{setup}\label{setup:polarized-canonical}
Let $\C\subseteq\F_2^n$ be a neural code on $n$ neurons.  Write $J_\C\subseteq R$ for its neural ideal,
and let
\[
I  =  \P(J_\C) \subseteq  S
\]
denote the polarized neural ideal.  When no confusion can arise, we write $\Delta=\Delta_\C$ for the
polar complex of $\C$, i.e.\ the Stanley--Reisner complex of $I$.
\end{setup}

\begin{remark}\label{rem:basic_degree_bounds}
Let $\C\subseteq\F_2^n$ with polarized neural ideal $I=\P(J_\C)$ and polar complex $\Delta=\Delta_\C$.
\begin{enumerate}[label=(\alph*)]
\item No minimal generator of $I$ is divisible by $x_i y_i$.

\item If $x_i,y_i\in I$, then depolarization gives $x_i, 1-x_i\in J_\C$. Hence $1\in J_\C$.
Thus $J_\C=R$ and $\C=V(J_\C)=\emptyset$.

\item The ideal $I$ is squarefree and is supported on a vertex set
$V(\Delta)\subseteq\{x_1,\dots,x_n,y_1,\dots,y_n\}$, so $|V(\Delta)|\le 2n$.
In particular, $\beta_{i,j}(I)=0$ for all $j>2n$.

\item Every minimal generator of $I$ is a \emph{transversal} monomial $x_\sigma y_\tau$ with
$\sigma\cap\tau=\varnothing$.  Hence $\deg(x_\sigma y_\tau)=|\sigma|+|\tau|\le n$, and therefore
$\beta_{0,j}(I)=0$ for all $j>n$.
\end{enumerate}
\end{remark}

\begin{remark}\label{rem:degenerate-codes}
Two degenerate cases will be used implicitly. 
We use the notation $2^{[n]}$ to denote the list of all codewords on $n$ neurons. 
\begin{enumerate}[label=(\alph*)]
\item If $\C= 2^{[n]}$, then $J_\C=(0)$ and $I=\P(J_\C)=(0)$.
\item If $\C=\varnothing$, then $J_\C=(1)$ and $I=\P(J_\C)=S$. Hence $\pd(I)=0$ and $\reg(I)=0$.
Indeed, since every $v\in\F_2^n$ is a noncodeword,
\[
1  =  \prod_{i=1}^n\bigl(x_i+(1-x_i)\bigr)  =  \sum_{v\in\F_2^n}\rho_v \in J_\C.
\]
\end{enumerate}
\end{remark}

\begin{remark}\label{rem:n=1}
When $n=1$, a nonempty code $\C\subseteq\F_2$ is either $\F_2$, $\{0\}$, or $\{1\}$.
If $\C=\F_2$ then $I=(0)$.  Otherwise, in $S=k[x_1,y_1]$,
\[
\P(J_{\{0\}})=\langle x_1\rangle
\qquad\text{and}\qquad
\P(J_{\{1\}})=\langle y_1\rangle,
\]
and in either case $\pd(I)=0$ and $\reg(I)=1$.
\end{remark}

\medskip

Throughout the paper we assume $n\ge 2$ and $\C\neq 2^{[n]}$ (equivalently, $I\neq (0)$).

\begin{definition}\label{def:nondegenerate}
We call $\C$ \emph{nondegenerate} if no coordinate is fixed on $\C$, i.e.\ for every $i\in[n]$ there
exist $c,c'\in\C$ with $c_i=1$ and $c'_i=0$.

Equivalently, in any realization $\U=\{U_1,\dots,U_n\}$ of $\C$, each $U_i$ is nonempty and does not
cover the entire stimulus space $X$.  By \Cref{thm:rf-neural-ideal}, this is equivalent to requiring
that neither $x_i\in J_\C$ nor $(1-x_i)\in J_\C$.  After polarization, this becomes:
\[
\C\ \text{is nondegenerate}\qquad\Longleftrightarrow\qquad
x_i\notin I\ \text{and}\ y_i\notin I\ \text{for all }i\in[n].
\]
\end{definition}

\subsection*{Free and constant neurons}

Let $R_{n+1}=R[x_{n+1}]$ and $S_{n+1}=S[x_{n+1},y_{n+1}]$.

\begin{lemma}\label{lem:free-neuron}
Let $\C\subseteq\F_2^n$ and define the free-neuron extension
\[
\C'  :=  \{c0, c1:\ c\in\C\} = \C\times\F_2  \subseteq \F_2^{n+1}.
\]
Let $I'=\P(J_{\C'})\subseteq S_{n+1}$.  Then $\pd(I')=\pd(I)$ and $\reg(I')=\reg(I)$.
\end{lemma}

\begin{proof}
The noncodewords of $\C'=\C\times\F_2$ are exactly $(\F_2^n\setminus\C)\times\F_2$.  So the neural ideal is
obtained by extension of scalars:
\[
J_{\C'}  =  J_\C\cdot R_{n+1}.
\]
Hence
\[
I'=\P(J_{\C'})  =  \P(J_\C)\cdot S_{n+1}  =  I\cdot S_{n+1}.
\]
The equalities $\pd(I')=\pd(I)$ and $\reg(I')=\reg(I)$  follow from
\Cref{lem:add_variable_pd_reg}.
\end{proof}

\begin{lemma}\label{lem:constant-neuron}
Let $\C\subseteq\F_2^n$ and define the constant-neuron extensions
\[
\C'_0:=\{c0:\ c\in\C\}\subseteq\F_2^{n+1}
\qquad\text{and}\qquad
\C'_1:=\{c1:\ c\in\C\}\subseteq\F_2^{n+1}.
\]
Let $I'_0=\P(J_{\C'_0})$ and $I'_1=\P(J_{\C'_1})$ in $S_{n+1}$.  Then
\[
\pd(I'_0)=\pd(I'_1)=\pd(I)+1
\qquad\text{and}\qquad
\reg(I'_0)=\reg(I'_1)=\reg(I).
\]
\end{lemma}

\begin{proof}
We treat $\C'_0$; the case $\C'_1$ is analogous.  Since every codeword in $\C'_0$ has last coordinate $0$,
we have $x_{n+1}\in J_{\C'_0}$.  Moreover, if $v\notin\C$, then both $(v,0)$ and $(v,1)$ are
noncodewords of $\C'_0$. So $J_\C\cdot R_{n+1}\subseteq J_{\C'_0}$.  Conversely, if $w\notin\C'_0$ then
either $w_{n+1}=1$ (so $\rho_w$ is divisible by $x_{n+1}$) or $w=(v,0)$ with $v\notin\C$ (so $\rho_w$ is
divisible by $\rho_v$).  Hence
\[
J_{\C'_0}  =  J_\C\cdot R_{n+1} + (x_{n+1}).
\]
Polarizing gives
\[
I'_0  =  \P(J_{\C'_0})
 =  \P(J_\C)\cdot S_{n+1} + (x_{n+1})
 =  I\cdot S_{n+1} + (x_{n+1}).
\]
Then \Cref{lem:add_variable_pd_reg} yields
$\pd(I'_0)=\pd(I)+1$ and $\reg(I'_0)=\reg(I)$.

For $\C'_1$, the same argument shows $J_{\C'_1}=J_\C\cdot R_{n+1}+(1-x_{n+1})$, and polarization sends
$(1-x_{n+1})$ to $y_{n+1}$, giving $I'_1=I\cdot S_{n+1}+(y_{n+1})$ and the same conclusions.
\end{proof}

\section{Polar complex and its connectedness}\label{subsec:polar-connectedness}

In this section we study the polar complex $\Delta$ and connectivity properties of its induced subcomplexes. These connectivity statements will be used repeatedly to force vanishing of certain homology groups in Hochster’s formula. In particular, we prove that $\Delta_\C$ is connected and that every induced subcomplex on sufficiently many vertices is connected.

\begin{lemma}\label{lem:vertex_size} 
   Let $\C\subseteq \F_2^n$ be a nonempty neural code with its polar complex $\Delta$. Then $|V(\Delta)|\ge n$. In particular, if $|V(\Delta)|=n$, then $\C$ consists of a single word and $\Delta$ is a simplex.
\end{lemma}

\begin{proof}
Fix $i\in[n]$. We claim that at least one of $x_i$ or $y_i$ is a vertex of $\Delta$.
Indeed, if neither $x_i$ nor $y_i$ were a vertex, then both $x_i$ and $y_i$ would lie in $I$.
By \Cref{rem:basic_degree_bounds}(2), this forces $\C=\emptyset$, contradicting the hypothesis.
Thus $|V(\Delta)|\ge n$.

Now assume $|V(\Delta)|=n$. Then for each $i$ exactly one of $\{x_i,y_i\}$ is a vertex of $\Delta$,
so the other variable lies in $I$. Depolarizing, we obtain that for each $i$ either $x_i\in J_\C$
or $(1-x_i)\in J_\C$. Since every polynomial in $J_\C$ vanishes on $\C$, it follows that for each $i$
the $i$th coordinate is constant across $\C$ (equal to $0$ in the first case and $1$ in the second).
Therefore $\C$ contains at most one word, and since $\C\neq\emptyset$ it consists of exactly one word.

Finally, when $|V(\Delta)|=n$ the ideal $I$ contains precisely one of $x_i,y_i$ for each $i$, hence is
generated by $n$ variables. Equivalently, $\Delta$ is the simplex on its vertex set.
\end{proof}

\begin{lemma}\label{lem:DW-connected-largeW}
Let $\C\subseteq \F_2^n$ be a nonempty neural code with its  polar complex $\Delta$ such that $|V(\Delta)|\ge n+1$.
If $W\subseteq V(\Delta)$ satisfies $|W|\ge n+1$, then the induced subcomplex $\Delta_W$ is connected.
\end{lemma}

\begin{proof}
Since there are only $n$ indices, the inequality $|W|\ge n+1$ implies that $W$ contains both vertices
from at least one pair $\{x_i,y_i\}$; that is, there exists $i\in[n]$ such that $x_i,y_i\in W$.
Because $x_i,y_i\in V(\Delta)$, we have $x_i,y_i\notin I$.
Moreover, no polar canonical form generator involves both $x_i$ and $y_i$, so $x_iy_i\notin I$.
Hence $\{x_i,y_i\}$ is an edge of $\Delta$, and therefore also an edge of the induced subcomplex $\Delta_W$.

Suppose for contradiction that $\Delta_W$ is disconnected, and let $\Gamma$ be the connected component
of $\Delta_W$ containing the edge $\{x_i,y_i\}$. Choose a vertex $u\in W\setminus V(\Gamma)$.
Then $u$ is not adjacent in $\Delta_W$ to either $x_i$ or $y_i$. So $\{u,x_i\}$ and $\{u,y_i\}$ are
nonfaces of $\Delta_W$, hence also nonfaces of $\Delta$.
Since $u,x_i,y_i\in V(\Delta)$, the singletons are faces, so these missing edges are minimal nonfaces.
Therefore the quadratic monomials $ux_i$ and $uy_i$ lie in $I$.

First assume $u=x_j$ for some $j$. Then $x_jx_i\in I$ and $x_jy_i\in I$.
Depolarizing $y_i\mapsto (1-x_i)$ gives $x_jx_i\in J_\C$ and $x_j(1-x_i)\in J_\C$. Hence
\[
x_j = x_jx_i + x_j(1-x_i)\in J_\C.
\]
Since $x_j$ is a pseudo-monomial and $1\notin J_\C$, it is minimal in $J_\C$
with respect to divisibility. Hence $x_j\in CF(J_\C)$ and $x_j\in I$, contradicting $x_j\in V(\Delta)$.

Now assume $u=y_j$ for some $j$. Then $y_jx_i\in I$ and $y_jy_i\in I$.
Depolarizing gives $x_i(1-x_j)\in J_\C$ and $(1-x_i)(1-x_j)\in J_\C$. Hence
\[
(1-x_j)=x_i(1-x_j)+(1-x_i)(1-x_j)\in J_\C.
\]
Again $(1-x_j)$ is minimal among pseudo-monomials in $J_\C$. Therefore
$(1-x_j)\in CF(J_\C)$ and $y_j\in I$, contradicting $y_j\in V(\Delta)$.

Therefore no such $u$ exists, and $\Delta_W$ is connected.
\end{proof}

\begin{corollary}\label{cor:polar-complex-connected}
Let $\C\subseteq \F_2^n$ be a nonempty neural code. Then its polar complex $\Delta$ is connected.
\end{corollary}

\begin{proof}
It follows from \Cref{lem:vertex_size}  that $|V(\Delta)|\geq n$ and $\Delta$ is connected when $|V(\Delta)|=n$.
If $|V(\Delta)|\ge n+1$,  Lemma~\ref{lem:DW-connected-largeW} applied to $W=V(\Delta)$
shows that $\Delta$ is connected.
\end{proof}
\section{Maximum Projective Dimension}\label{sect:results}

In this section we study the projective dimension of polarized neural ideals. Our first goal is to provide an upper bound on
$\pd(I)$ in terms of $n$, and our second goal is to determine whether this bound is sharp.

\begin{remark}\label{rem:degenerate}
Since $I\subseteq S$ is an ideal in a polynomial ring in $2n$ variables, Hilbert--Syzygy Theorem gives the crude bound
$\pd(I)\le 2n$.  If $I\neq 0$ (namely, $\C\neq \F_2^n$), then  
% any variable (e.g.\ $x_1$) is a
% nonzerodivisor on $I$, so 
$\depth(I)\ge 1$.  By Auslander--Buchsbaum,
\[
\pd (I)\le 2n-1.
\]
\end{remark}

We can improve the upper bound as follows. 

\begin{theorem}\label{thm:pdim_upper_bound}
Let $\C\subseteq\F_2^n$ be a nonempty neural code. Then
\[
\pd(I) \le 2n-3.
\]
\end{theorem}

\begin{proof}
Recall that $V(\Delta)\subseteq V$ where $V=\{x_1,\dots,x_n,y_1,\dots,y_n\}$.
By Hochster’s formula, if $\beta_{i,j}(I)\neq 0$ then there exists $W\subseteq V(\Delta)$ with $|W|=j$
such that $\widetilde H_{j-i-2}(\Delta_W;k)\neq 0$.

Since $|V(\Delta)|\le 2n$, we have $j\le 2n$. If $i\ge 2n-2$, then $j-i-2\le 0$.
Moreover, for $|W|\ge 1$ the induced complex $\Delta_W$ is nonempty. So $\widetilde H_t(\Delta_W;k)=0$
for all $t<0$. Hence the only potentially nonzero case is when $(i,j)=(2n-2,2n)$.

If $|V(\Delta)|<2n$, then no subset $W\subseteq V(\Delta)$ has $|W|=2n$, so $\beta_{2n-2,2n}(I)=0$.
If instead $|V(\Delta)|=2n$, then the only such subset is $W=V(\Delta)$ and 
$\beta_{2n-2,2n}(I)=\dim_k \widetilde H_0(\Delta;k)$.
By Corollary~\ref{cor:polar-complex-connected}, $\Delta$ is connected. So $\widetilde H_0(\Delta;k)=0$.

Therefore $\beta_{i,j}(I)=0$ for all $i\ge 2n-2$. Hence  $\pd(I)\le 2n-3$.
\end{proof}

We next relate the extremal projective dimension to regularity: attaining the upper bound
$\pd(I)=2n-3$ forces a lower bound on $\reg (I)$.

\begin{lemma}\label{lem:pdim-max-implies-reg-ge-3}
If \(\pd(I)=2n-3\), then $\beta_{2n-3,2n} (I)\neq 0$. In particular, \(\reg(I)\ge 3\) in this case.
\end{lemma}

\begin{proof}
Since $\pd(I)=2n-3$, there exists $j$ with
$\beta_{2n-3,j}(I)\neq 0$. By Hochster’s formula, there is a subset $W\subseteq V(\Delta)$ with $|W|=j$
such that $\widetilde H_{j-2n+1}(\Delta_W;k)\neq 0$.

Because $W\subseteq V(\Delta)$, the induced complex $\Delta_W$ is nonempty. Hence
$\widetilde H_t(\Delta_W;k)=0$ for $t<0$. Therefore $j-2n+1\ge 0$, so $j\ge 2n-1$.
On the other hand $j\le |V(\Delta)|\le 2n$. Thus either $j=2n-1$ or $j=2n$.

If $j=2n-1$, then Hochster’s formula forces $\widetilde H_0(\Delta_W;k)\neq 0$, so $\Delta_W$ is disconnected.
But $|W|=2n-1\ge n+1$, and \Cref{lem:DW-connected-largeW} asserts that every induced subcomplex on at least $n+1$
vertices is connected, a contradiction. Hence $j\neq 2n-1$. Therefore $j=2n$ and
\[
\beta_{2n-3,2n}(I)\neq 0.
\]
Consequently, $\reg(I)\ge j-(2n-3)=3$.
\end{proof}

Next, we identify classes of neural codes achieving the upper bound from \Cref{thm:pdim_upper_bound}.

\begin{theorem}\label{thm:all-or-nothing}
Let $\C_0=\{00\cdots 0, 11\cdots 1\}\subseteq\F_2^n$. Then
\[
\pd(\P(J_{\C_0}))=2n-3
\qquad\text{and}\qquad
\reg(\P(J_{\C_0}))=3.
\]
\end{theorem}

We call $\C_0$  \emph{``all-or-nothing code"} on $n$ neurons since all the neurons either fire simultaneously or not.

\begin{figure}[ht]
        \centering
        \includegraphics[width=5cm]{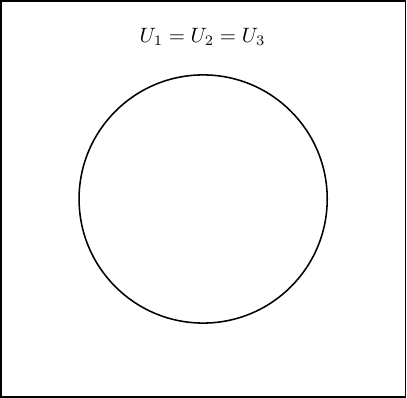}   
        \label{fig: A20}
    \caption{Realization  of all-or-nothing code $\C_0  = \{000, 111\}$ for $n=3$}
\end{figure}

\begin{proof}
Write $J=J_{\C_0}$ and $I=\P(J)$.
Since $\C_0=\{0^n,1^n\}$, every noncodeword $w\notin\C_0$ has at least one coordinate equal to $1$ and
at least one coordinate equal to $0$. Consequently, for each $w\notin\C_0$ the characteristic
pseudo-monomial $\rho_w$ is divisible by $x_i(1-x_j)$ for some ordered pair $i\neq j$ with
$w_i=1$ and $w_j=0$. Hence
\[
J=\langle \rho_w : w\notin\C_0\rangle \subseteq \langle x_i(1-x_j) : i\neq j\rangle.
\]

For the reverse inclusion, fix $i\neq j$ and sum $\rho_w$ over all words with $(w_i,w_j)=(1,0)$:
\[
\sum_{\substack{w:~~ w_i=1,~ w_j=0}}\rho_w
=
x_i(1-x_j)\prod_{k\neq i,j}\bigl(x_k+(1-x_k)\bigr)
=
x_i(1-x_j).
\]
Since each summand $\rho_w$ corresponds to a noncodeword, $x_i(1-x_j)\in J$ for all $i\neq j$. Therefore
\[
J=\langle x_i(1-x_j): i\neq j\rangle.
\]
In particular, these generators are minimal: no linear pseudo-monomial lies in $J$ because $\C_0$ is
nondegenerate, and $x_i(1-x_j)$ has only linear pseudo-monomial divisors. Hence
$CF(J)=\{x_i(1-x_j): i\neq j\}$. Polarizing sends $(1-x_j)$ to $y_j$, so
\[
I=\P(J)=\langle x_i y_j : i\neq j\rangle.
\]
Thus $I$ is the edge ideal of the bipartite graph $G$ on
$\{x_1,\dots,x_n\}\sqcup\{y_1,\dots,y_n\}$ with all edges present except the $n$ matching edges
$\{x_i,y_i\}$.
Equivalently, the Stanley--Reisner complex $\Delta=\Delta_{\C_0}$ has two $(n-1)$-simplices
$[x_1,\dots,x_n]$ and $[y_1,\dots,y_n]$, and also the $n$ edges $[x_i,y_i]$, each of which is a maximal
face.

We now show $\pd(I)=2n-3$. The upper bound $\pd(I)\le 2n-3$ is given by \Cref{thm:pdim_upper_bound}.
For the reverse inequality, it suffices to show $\beta_{2n-3,2n}(I)\neq 0$. By Hochster's formula,
\[
\beta_{2n-3,2n}(I)=\dim_k\widetilde H_{2n-(2n-3)-2}(\Delta;k)=\dim_k\widetilde H_{1}(\Delta;k).
\]
So it is enough to find a nonzero class in $\widetilde H_1(\Delta;k)$.

Fix distinct $i,j$ and consider the induced subcomplex on
$W=\{x_i,x_j,y_i,y_j\}$. In $\Delta_W$ the edges
\[
[x_i,x_j],\ [x_j,y_j],\ [y_j,y_i],\ [y_i,x_i]
\]
are present, while the cross-edges $[x_i,y_j]$ and $[x_j,y_i]$ are absent because
$x_i y_j$ and $x_j y_i$ are generators of $I$. Hence $\Delta_W$ is a $4$-cycle with no $2$-faces. Let
$z\in C_1(\Delta;k)$ be the corresponding $1$-cycle
\[
z=[x_i,y_i]+[y_i,y_j]+[y_j,x_j]+[x_j,x_i].
\]
Then $z\in\ker\partial_1$. Moreover, every $2$-simplex of $\Delta$ lies entirely in the $x$-simplex
or entirely in the $y$-simplex. In particular, the boundary of any $2$-chain is supported only on edges
among the $x$'s or only on edges among the $y$'s. Since $z$ contains the mixed edges $[x_i,y_i]$ and
$[x_j,y_j]$ with nonzero coefficients, it cannot lie in $\mathrm{im}(\partial_2)$. Thus $[z]\neq 0$ in
$\widetilde H_1(\Delta;k)$. So $\beta_{2n-3,2n}(I)\neq 0$ and $\pd(I)=2n-3$.

For regularity, since $I=I(G)$ is the edge ideal of a connected bipartite graph, we apply
\cite[Theorem~3.1]{fernándezramos}: one has $\reg I(G)=3$ if and only if the complement $G^c$ contains an
induced cycle of length at least $4$, and the bipartite complement $G^{bc}$ contains no induced cycle of
length at least $6$. Here $G^c$ contains the induced $4$-cycle
\[
x_i - x_j - y_j - y_i - x_i \qquad (i\neq j),
\]
and $G^{bc}$ is exactly the perfect matching $\{\{x_t,y_t\}:t\in[n]\}$, hence has no induced cycle of
length $\ge 6$. Therefore $\reg(I)=3$.
\end{proof}

All-or-nothing codes are a special case of antipodal pairs. We generalize the pattern observed in all-or-nothing codes as follows.

\begin{definition}\label{def:antipodal_pair}
For $v\in \F_2^n$, let $\overline v$ denote its bitwise complement.
We call
\[
\C_v  :=  \{v,\overline v\}\subseteq \F_2^n
\]
an \emph{antipodal pair code}.
\end{definition}

\begin{corollary}\label{cor:antipodal_pairs_maxpdim}
Let $\C_v=\{v,\bar v\}\subseteq\F_2^n$. Then
\[
\pd(\P(J_{\C_v}))=2n-3
\qquad\text{and}\qquad
\reg(\P(J_{\C_v}))=3.
\]
\end{corollary}

\begin{proof}
Fix $i\neq j$. The restriction of $\C_v=\{v,\bar v\}$ to $\{i,j\}$ is
$\{(v_i,v_j),(1-v_i,1-v_j)\}$. Hence exactly two patterns on $\{i,j\}$ are missing. By
\Cref{thm:rf-neural-ideal}, the corresponding quadratic pseudo-monomials lie in $J_{\C_v}$, namely
\begin{equation}\label{eq:antipodal}\tag{$\star$}
\begin{aligned}
v_i=v_j \ &\Longrightarrow\ x_i(1-x_j),\ x_j(1-x_i)\in J_{\C_v},\\
v_i\neq v_j \ &\Longrightarrow\ x_ix_j,\ (1-x_i)(1-x_j)\in J_{\C_v}.
\end{aligned}
\end{equation}

(Equivalently, each can be obtained by summing $\rho_w$ over all words with the relevant missing
two-coordinate pattern, as in \Cref{thm:all-or-nothing}.)

Now let $w\notin\C_v$. Since $w\neq v,\bar v$, choose $i$ with $w_i=v_i$ and $j$ with $w_j\neq v_j$.
Then the restriction of $w$ to $\{i,j\}$ is one of the missing patterns above. So $\rho_w$ is divisible
by the corresponding  pseudo-monomial from \eqref{eq:antipodal}. Hence these pseudo-monomials in  \eqref{eq:antipodal} generate $J_{\C_v}$. Since $\C_v$ is
nondegenerate, no linear pseudo-monomial lies in $J_{\C_v}$. So this generating set is minimal under
divisibility and therefore equals $CF(J_{\C_v})$.

Polarizing yields that $\P(J_{\C_v})$ is generated by the corresponding quadrics
$x_i y_j,\ x_j y_i,\ x_ix_j,\ y_i y_j$.  Set $T=\{i\in[n]:v_i=1\}$. Define
\[
a_i=\begin{cases}x_i,& i\notin T,\\ y_i,& i\in T,\end{cases}
\qquad
b_i=\begin{cases}y_i,& i\notin T,\\ x_i,& i\in T.\end{cases}
\]
Then, we have  $\P(J_{\C_v})=\langle a_i b_j : i\neq j\rangle$ since 
\[
a_i b_j=
\begin{cases}
x_i y_j, & i\notin T,\ j\notin T,\\
x_i x_j, & i\notin T,\ j\in T,\\
y_i y_j, & i\in T,\ j\notin T,\\
y_i x_j, & i\in T,\ j\in T.
\end{cases}
\]
Let $A=\{a_1,\dots,a_n\}$ and $B=\{b_1,\dots,b_n\}$. Then $\P(J_{\C_v})$ is the edge ideal of the bipartite graph on $A\sqcup B$ with all edges except the
matching edges $\{a_i,b_i\}$, i.e.\ the same graph as in \Cref{thm:all-or-nothing} after relabeling
vertices.  Therefore $\P(J_{\C_v})$ has the same graded Betti numbers as $\P(J_{\C_0})$, and the stated
values of $\pd$ and $\reg$ follow from \Cref{thm:all-or-nothing}.
\end{proof}

\begin{remark}\label{rem:many_antipodal_examples}
There are $2^{n-1}$ distinct antipodal pair codes $\{v,\overline v\}$ on $n$ neurons.
Thus \Cref{cor:antipodal_pairs_maxpdim} provides exponentially many examples of neural codes
whose polarized neural ideals attain the maximal projective dimension $2n-3$ while having regularity $3$.
\end{remark}

We end this section with the following question.

\begin{question}\label{q:classify_maxpdim}
Classify all neural codes on $n$ neurons whose polarized neural ideals attain the maximal projective dimension $2n-3$.
\end{question}
\section{Maximum Regularity}\label{sec:regularity}

In this section we study the regularity of polarized neural ideals. Our goal is to describe the range of possible values of  $\reg(I)$, with particular emphasis on establishing a sharp upper bound and characterizing the neural codes for which this bound is attained.

\begin{lemma}\label{lem:reg_upper_bound}
Let \(\C \subseteq \F_2^n\) be a neural code on \(n\) neurons.
Then
\[
\reg(I)\le 2n-1.
\]
\end{lemma}

\begin{proof}
Recall from Remark~\ref{rem:basic_degree_bounds} (b) that  \(j\le 2n\) when \(\beta_{i,j}(I)\neq 0\). So \(j-i\le 2n-1\)  when $i\geq 1$.
For \(i=0\), Remark~\ref{rem:basic_degree_bounds} (c) gives \(\beta_{0,j}(I)=0\) for \(j>n\).
Hence no nonzero Betti number can occur with \(j-i>2n-1\). Therefore \(\reg(I)\le 2n-1\).
\end{proof}

The next proposition pins down where the extremal value \(\reg(I)=2n-1\) must appear in the Betti table.

\begin{proposition}\label{prop:regmax_beta12n}
If \(\reg(I)=2n-1\), then \(\beta_{1,2n}(I)\neq 0\).
\end{proposition}

\begin{proof}
Assume \(\reg(I)=2n-1\). Then there exist \(i,j\) with \(\beta_{i,j}(I)\neq 0\) and \(j-i=2n-1\).
By Remark~\ref{rem:basic_degree_bounds} (b), we have \(j\le 2n\), so \(i\le 1\).
If \(i=0\), then \(j=2n-1>n\), contradicting Remark~\ref{rem:basic_degree_bounds} (c).
Thus \((i,j)=(1,2n)\) and \(\beta_{1,2n}(I)\neq 0\).
\end{proof}

\begin{lemma}\label{lem:reg-max-antipodal-noncodewords}
If \(\reg(I)=2n-1\), then \(\C\) omits an antipodal pair \(\{v,\overline v\}\subseteq\F_2^n\).
\end{lemma}

\begin{proof}
Assume \(\reg(I)=2n-1\). By \Cref{prop:regmax_beta12n} we have
\(\beta_{1,2n}(I)\neq 0\). Hence there exists a squarefree monomial \(u\) of degree \(2n\) with
\(\beta_{1,u}(I)\neq 0\). So, \(u=x_1\cdots x_n y_1\cdots y_n\). For monomial ideals, a multidegree \(u\) occurring in homological degree \(1\)
is the least common multiple of two minimal generators. Thus there are minimal generators \(m_1,m_2\in I\)
with \(u=\lcm(m_1,m_2)\).

Recall that  every minimal generator of \(I\) is transversal
(no generator is divisible by \(x_i y_i\)). Therefore for each \(i\in[n]\), exactly one of \(x_i,y_i\)
divides \(m_1\), and \(m_2\) contains the other. In particular \(\deg(m_1)=\deg(m_2)=n\), and there is
\(\sigma\subseteq[n]\) such that
\[
m_1=x_\sigma y_{[n]\setminus\sigma},
\qquad
m_2=x_{[n]\setminus\sigma} y_\sigma.
\]

Let \(v\in\F_2^n\) be defined by \(v_i=1\) iff \(i\in\sigma\). Then \(m_1=\P(\rho_v)\) and
\(m_2=\P(\rho_{\overline v})\). Since \(\rho_w\in J_\C\) exactly when \(w\notin\C\), it follows that
\(v,\overline v\notin\C\).
\end{proof}

\begin{lemma}\label{lem:min-charpoly_implies_neighbors}
Let \(v\in\F_2^n\) and write \(\sigma=\{i:v_i=1\}\).
Suppose \(\rho_v=x_\sigma\prod_{j\notin\sigma}(1-x_j)\) is a minimal pseudo-monomial in \(J_\C\)
(i.e.\ no proper divisor of \(\rho_v\) lies in \(J_\C\)).
Then \(v\notin\C\), and every codeword at Hamming distance \(1\) from \(v\) belongs to \(\C\).
\end{lemma}

\begin{proof}
Since \(\rho_v\in J_\C\) and \(\rho_v(v)=1\), we have \(v\notin V(J_\C)=\C\).

Fix an index \(t\in[n]\). Define
\[
f_t  := 
\begin{cases}
\rho_v/x_t, & \text{if } t\in\sigma  (v_t=1),\\[2pt]
\rho_v/(1-x_t), & \text{if } t\notin\sigma  (v_t=0).
\end{cases}
\]
Then \(f_t\) is a proper divisor of \(\rho_v\) and by minimality \(f_t\notin J_\C\).
Observe that \(f_t\notin J_\C\)
implies that there exists \(c^{(t)}\in\C\) on which \(f_t\) evaluates to \(1\) which means \(c\) satisfies the coordinate
constraints encoded by \(f_t\). Equivalently, \(c^{(t)}_i=v_i\) for all \(i\neq t\). Since \(v\notin\C\), we have \(c^{(t)}\neq v\),
and therefore \(c^{(t)}\) differs from \(v\) exactly in the \(t\)-th coordinate. Hence \(c^{(t)}\) is obtained from \(v\) by flipping the \(t\)-th bit. Thus \(c^{(t)}\) is a Hamming neighbor of \(v\) that lies in \(\C\).
As \(t\) was arbitrary, every Hamming neighbor of \(v\) belongs to \(\C\).
\end{proof}

\begin{corollary}\label{cor:regmax_neighbors}
Suppose \(\reg(I)=2n-1\). Then \(\C\) omits an antipodal pair \(\{v,\overline v\}\),
and every Hamming neighbor of \(v\) and of \(\overline v\) lies in \(\C\).
\end{corollary}

\begin{proof}
It follows from \Cref{lem:reg-max-antipodal-noncodewords} and its proof that \(\C\) omits an antipodal pair \(\{v,\overline v\}\) and
\(\P(\rho_v),\P(\rho_{\overline v})\) are minimal generators of \(I\).
In particular, \(\rho_v\) and \(\rho_{\overline v}\) are minimal pseudo-monomials in \(J_\C\). 
Otherwise some \(f\in CF(J_\C)\) would properly divide \(\rho_v\) (respectively \(\rho_{\overline v}\)),
and then \(\P(f)\) would properly divide \(\P(\rho_v)\) (respectively \(\P(\rho_{\overline v})\)),
contradicting minimality in \(I\).
Then applying \Cref{lem:min-charpoly_implies_neighbors} to \(\rho_v\) and \(\rho_{\overline v}\) completes the proof.
\end{proof}

We now show that the upper bound in \Cref{lem:reg_upper_bound} is sharp, and we classify all codes achieving it.

\begin{theorem}\label{thm:maxreg_classification}
Let \(\C\subseteq \F_2^n\) be a nonempty neural code.
Then the following are equivalent:
\begin{enumerate}
\item \(\reg(I)=2n-1\);
\item \(I\) is minimally generated by two complementary degree-\(n\) monomials
\[
I=\langle x_\sigma y_{\sigma^c},  x_{\sigma^c}y_\sigma\rangle
\]
for some \(\sigma\subseteq [n]\);
\item \(\C=\F_2^n\setminus \{v,\overline v\}\) for some \(v\in\F_2^n\).
\end{enumerate}
In particular, if \(\reg(I)=2n-1\) then \(\pd(I)=1\).
\end{theorem}

\begin{figure}[ht]
   \centering
        \includegraphics[width = 5cm]{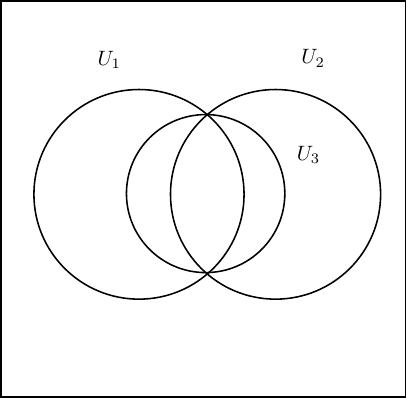}
        \caption{Realization of $\C = \F_2^3 \setminus \{ 001, 110\}$}
        \label{fig : A10}
 \end{figure}

 \begin{proof}
Set \(V=\{x_1,\dots,x_n,y_1,\dots,y_n\}\).

\((2)\Rightarrow(1)\):
Assume \(I=\langle m_1,m_2\rangle\) with \(m_1=x_\sigma y_{\sigma^c}\) and \(m_2=x_{\sigma^c}y_\sigma\).
Then \(\gcd(m_1,m_2)=1\) and the Taylor resolution of $I$ is minimal. 
Hence \(\reg(I)=\max\{ n, 2n-1 \}=2n-1\) and \(\pd(I)=1\).

\((1)\Rightarrow(2)\):
Assume \(\reg(I)=2n-1\). By \Cref{prop:regmax_beta12n} we have \(\beta_{1,2n}(I)\neq 0\), and by
\Cref{lem:reg-max-antipodal-noncodewords} there exist complementary minimal generators
\[
m_1=x_\sigma y_{\sigma^c},\qquad m_2=x_{\sigma^c}y_\sigma
\]
for some \(\sigma\subseteq[n]\). Let \(A=\supp(m_1)\) and \(B=\supp(m_2)\). Then \(V=A\sqcup B\) and
\(|A|=|B|=n\). In particular,  \(V(\Delta)=V\) for the Stanley--Reisner complex \(\Delta\) with \(I=I_\Delta\).
Since \(m_1\) and \(m_2\) are minimal generators, \(A\) and \(B\) are minimal nonfaces of \(\Delta\).
Thus no face of \(\Delta\) can contain all of \(A\) or all of \(B\).  Equivalently,
\[
\Delta \subseteq \Delta_0  := \{F\subseteq V : A\nsubseteq F,  B\nsubseteq F\}.
\]
Let $\Delta_A$ and $\Delta_B$ be the full simplices on $A$ and $B$.  Then
\[
\Delta_0=\partial(\Delta_A)*\partial(\Delta_B),
\]
the join of the boundary spheres of $\Delta_A$ and $\Delta_B$. Its facets are exactly
\(F_{a,b}:=V\setminus\{a,b\}\) with \(a\in A\) and \(b\in B\).
In particular, \(\Delta_0\) is a simplicial \((2n-3)\)-sphere, so
\(\widetilde H_{2n-3}(\Delta_0;k)\cong k\).

\begin{claim}\label{clm:topcycle-uses-all-facets-regsec}
Every \((2n-3)\)-cycle in \(\Delta_0\) (with coefficients in \(k\)) is a scalar multiple of
\(\sum_{a\in A,  b\in B} F_{a,b}\). In particular, any nonzero \((2n-3)\)-cycle in \(\Delta_0\)
has nonzero coefficient on every facet \(F_{a,b}\).
\end{claim}

\begin{proof}[Proof of Claim~\ref{clm:topcycle-uses-all-facets-regsec}]
Fix an orientation of the simplicial sphere \(\Delta_0\), so that each ridge is contained in two
facets with opposite induced orientations. Write a \((2n-3)\)-chain as
\(z=\sum_{a\in A,  b\in B} c_{a,b} F_{a,b}\).
Fix \(b\in B\) and distinct \(a,a'\in A\). The ridge \(R=V\setminus\{a,a',b\}\) lies in exactly two facets,
\(F_{a,b}\) and \(F_{a',b}\), and their induced orientations on \(R\) are opposite. Hence the coefficient
of \(R\) in \(\partial z\) is \(\pm(c_{a,b}-c_{a',b})\). If \(\partial z=0\), then \(c_{a,b}=c_{a',b}\).
Thus for fixed \(b\), the coefficient \(c_{a,b}\) is independent of \(a\). A symmetric argument (fix \(a\)
and vary \(b\)) shows the coefficients are also independent of \(b\). Therefore all \(c_{a,b}\) are equal
to a common scalar \(\lambda\), and the claim follows.
\end{proof}

If \(\Delta\neq\Delta_0\), then \(\Delta\) omits some facet \(F_{a,b}\) of \(\Delta_0\).
Let \(z\) be any \((2n-3)\)-cycle in \(\Delta\). Viewing \(z\) as a \((2n-3)\)-chain in \(\Delta_0\)
(by assigning coefficient \(0\) to every facet of \(\Delta_0\setminus\Delta\)), we obtain a
\((2n-3)\)-cycle in \(\Delta_0\) whose coefficient on \(F_{a,b}\) is \(0\). By
Claim~\ref{clm:topcycle-uses-all-facets-regsec}, this forces \(z=0\). Hence
\(\widetilde H_{2n-3}(\Delta;k)=0\).
Applying Hochster's formula at \((i,j)=(1,2n)\),
and using that the only subset \(W\subseteq V\) with \(|W|=2n\) is \(W=V\), we obtain
\[
\beta_{1,2n}(I)=\dim_k \widetilde H_{2n-3}(\Delta;k)=0,
\]
contradicting \(\beta_{1,2n}(I)\neq 0\). Therefore \(\Delta=\Delta_0\). So \(A\) and \(B\) are the only minimal
nonfaces of \(\Delta\). Equivalently, \(I\) has exactly the two minimal generators \(m_1,m_2\), proving~(2).

\((2)\Rightarrow(3)\):
Assume \(I=\langle x_\sigma y_{\sigma^c}, x_{\sigma^c}y_\sigma\rangle\) minimally and write
\(m_1=x_\sigma y_{\sigma^c}\), \(m_2=x_{\sigma^c}y_\sigma\).
By definition of polarization and canonical form, \(I=\langle \P(f): f\in CF(J_\C)\rangle\), and each
\(\P(f)\) is a transversal squarefree monomial of degree at most \( n\).
Since \(m_1,m_2\) are the only minimal monomial generators of \(I\), every \(\P(f)\) is divisible by \(m_1\)
or \(m_2\). The degree bound \(\deg \P(f)\le n=\deg(m_i)\) forces \(\P(f)\in\{m_1,m_2\}\).
Hence \(CF(J_\C)=\{\rho_v,\rho_{\overline v}\}\), where \(v_i=1\) iff \(i\in\sigma\).

Now let \(w\notin\C\). Then \(\rho_w\in J_\C\). Among the pseudo-monomial divisors of \(\rho_w\) that lie in \(J_\C\),
choose one minimal under divisibility. This divisor lies in \(CF(J_\C)\), so it is \(\rho_v\) or \(\rho_{\overline v}\),
and in particular it divides \(\rho_w\). Since characteristic pseudo-monomials contain exactly one factor from
\(\{x_i,1-x_i\}\) in each coordinate, divisibility \(\rho_u\mid \rho_w\) forces \(u=w\). Therefore
\(w\in\{v,\overline v\}\), and \(\C=\F_2^n\setminus\{v,\overline v\}\).

\((3)\Rightarrow(2)\):
If \(\C=\F_2^n\setminus\{v,\overline v\}\), then \(CF(J_\C)=\{\rho_v,\rho_{\overline v}\}\),
and polarization gives
\(
I=\langle \P(\rho_v),\P(\rho_{\overline v})\rangle
=\langle x_\sigma y_{\sigma^c}, x_{\sigma^c}y_\sigma\rangle,
\)
where \(\sigma=\{i:v_i=1\}\).

The final statement \(\pd(I)=1\) follows from \((2)\Rightarrow(1)\).
\end{proof}

\section{Small projective dimension: $\pd (I)\leq 1$}\label{sec:small-pdim}

In this section we study polarized neural ideals of small projective dimension. We first classify the codes
$\C\subseteq\F_2^n$ for which $\pd(I)=0$, and then give a family of codes with $\pd(I)=1$ whose regularity
ranges across $[1,2n-1]$.

\subsection{Classification of neural codes with projective dimension $0$}\label{sec:pdim0}

We begin by classifying neural codes $\C\subseteq\F_2^n$ whose polarized neural ideal $I$ has projective
dimension $0$.

\begin{theorem}\label{prop:pdim0_classification}
Let $\C\subseteq \F_2^n$ be a neural code. 
Then the following are equivalent:
\begin{enumerate}
\item $\pd(I)=0$;
\item $I$ is principal;
\item there exist disjoint subsets $\sigma,\tau\subseteq[n]$ with $\sigma\cup\tau\neq\emptyset$ such that
\[
I=\langle x_\sigma y_\tau\rangle;
\]
\item there exist disjoint $\sigma,\tau\subseteq[n]$ with $\sigma\cup\tau\neq\emptyset$ such that $\C=\F_2^n\setminus Q(\sigma,\tau)$ where
\[
Q(\sigma,\tau):=\{v\in \F_2^n:\ v_i=1\ (i\in\sigma),\ v_j=0\ (j\in\tau)\}.
\]
\end{enumerate}
Moreover, writing $r:=|\sigma|+|\tau|$, we have $1\le r\le n$ and $\reg(I)=r$.
\end{theorem}

\begin{proof}
$(1)\Leftrightarrow(2)$:
A finitely generated graded $S$--module has projective dimension $0$ if and only if it is free, i.e. $I\cong S(-d)$ for some $d$ which is equivalent to  $I$ is principal.

$(2)\Leftrightarrow(3)$:
Assume $I=\langle m\rangle$ is principal.
Then the unique minimal monomial generator $m$  is transversal in the sense of \Cref{def:transversal_monomial} (squarefree and $x_i y_i\nmid m$ for all $i\in[n]$).
Writing
\[
\sigma:=\{i : x_i \mid m\},\qquad \tau:=\{j : y_j \mid m\},
\]
we have $m=x_\sigma y_\tau$ with $\sigma\cap\tau=\emptyset$.
The implication $(3)\Rightarrow(2)$ is immediate.

$(3)\Rightarrow(4)$:
If $I=\langle x_\sigma y_\tau\rangle$, then depolarizing gives
\[
J_\C=\langle f\rangle
\qquad\text{where}\qquad
f=x_\sigma\prod_{j\in\tau}(1-x_j).
\]
For $v\in\F_2^n$ we have
\[
f(v)=\prod_{i\in\sigma} v_i \cdot \prod_{j\in\tau} (1-v_j)\in\{0,1\}.
\]
So $f(v)=1$ if and only if $v\in Q(\sigma,\tau)$ and $f(v)=0$ otherwise. Since $V(J_\C)=\C$,
we obtain
\[
\C=\{v\in\F_2^n:\ f(v)=0\}=\F_2^n\setminus Q(\sigma,\tau).
\]

$(4)\Rightarrow(3)$:
Assume $\C=\F_2^n\setminus Q(\sigma,\tau)$ and set $f=x_\sigma\prod_{j\in\tau}(1-x_j)$.
If $v\in Q(\sigma,\tau)$, then $\rho_v$ is divisible by $f$.  Hence $J_\C\subseteq\langle f\rangle$.
Conversely,
\[
\sum_{v\in Q(\sigma,\tau)} \rho_v
=
x_\sigma\prod_{j\in\tau}(1-x_j)\prod_{\ell\notin\sigma\cup\tau}\bigl(x_\ell+(1-x_\ell)\bigr)
=f.
\]
So $f\in J_\C$ and therefore $J_\C=\langle f\rangle$.
Moreover, $f$ is minimal under divisibility among pseudo-monomials in $J_\C$ (dropping any factor yields a
pseudo-monomial that does not vanish on some codeword). Hence $CF(J_\C)=\{f\}$ and $J_\C=\langle f\rangle$ is in canonical form.
Polarizing gives $I=\langle x_\sigma y_\tau\rangle$.

Finally, if $I=\langle x_\sigma y_\tau\rangle$ is principal generated in degree
$r:=|\sigma|+|\tau|$, then its minimal free resolution is $0\to S(-r)\to I\to 0$, so $\reg(I)=r$.
Since $\sigma\cap\tau=\varnothing$ and $I\neq 0$, we have $1\le r\le n$.
\end{proof}

\subsection{The $\pd (I)=1$ line}\label{sec:pdim1}
We now turn our attention to the line $\pd (I)=1$.  In the  family we construct, the polarized neural ideal is generated by two
squarefree monomials, and hence has a resolution with a single first syzygy.  In this situation
$\reg$ is determined by the degree of the least common multiple of the two generators, which admits
a direct interpretation in terms of the Hamming distance between the corresponding missing words.
Using the free-neuron extension from \Cref{lem:free-neuron}, we realize every pair
$(1,r)$ with $r\in [1, 2n-1]$.

\begin{proposition}\label{prop:pdim1-any-reg}
Fix $n\ge 3$ and let $r$ be an integer with $1\le r\le 2n-1$.
There exists a neural code $\C_r\subseteq\F_2^n$ such that, with
$I_r=\P(J_{\C_r})$, we have
\[
(\pd(I_r),\reg(I_r))=(1,r).
\]
Moreover, if $r\ge 3$, the code $\C_r$ may be chosen to be nondegenerate.
\end{proposition}

\begin{proof}
For $r=1$, 
% nondegeneracy is impossible by \Cref{thm:reg1-classification}.  For $n\ge 3$, 
set
\[
\C_1=\{c\in\F_2^n:\ c_1=0,\ c_2=1\}.
\]
Then $\P(J_{\C_1})=\langle x_1,\ y_2\rangle$ and  $(\pd,\reg)=(1,1)$.

For $r=2$, set
\[
\C_2=\{c\in\F_2^n:\ c_1=0 \text{ and not }(c_2=c_3=1)\}.
\]
Then $\P(J_{\C_2})=\langle x_1,\ x_2x_3\rangle$. Hence $(\pd,\reg)=(1,2)$.  (One can also realize $(1,2)$ with a nondegenerate simplicial code; see
\Cref{sec:reg2-line} for related constructions.)

Now assume $r\ge 3$. Choose integers
\[
m:=\Big\lceil\frac{r+1}{2}\Big\rceil \le n,
\qquad
d:=r-m+1,
\]
so that $2\le d\le m$ and $m+d-1=r$. Pick distinct $v,w\in\F_2^m$ with $d_H(v,w)=d$ and define
\[
\C^{(m)}:=\F_2^m\setminus\{v,w\}.
\]
Since the noncodewords of $\C^{(m)}$ are exactly $v$ and $w$, we have
$J_{\C^{(m)}}=\langle \rho_v,\rho_w\rangle$.
We claim that $\rho_v$ and $\rho_w$ are precisely the elements of $CF(J_{\C^{(m)}})$.
Indeed, $d_H(v,w)\ge 2$ implies that for each of $v$ and $w$ there is a word of Hamming distance $1$
from it that is still a codeword (e.g.\ flip a single coordinate of $v$; the result is neither $v$ nor
$w$ since $w$ differs from $v$ in at least two coordinates).  Consequently, if $f$ is a proper
pseudo-monomial divisor of $\rho_v$ (so $\deg(f)<m$), then $f$ does not vanish on that neighboring
codeword, and hence $f\notin J_{\C^{(m)}}$.  Thus $\rho_v$ is minimal under divisibility among
pseudo-monomials in $J_{\C^{(m)}}$, and similarly for $\rho_w$.  Therefore $CF(J_{\C^{(m)}})=\{\rho_v,\rho_w\}$. It follows that
\[
I_{(m)}:=\P(J_{\C^{(m)}})=\langle m_v,m_w\rangle,
\]
where $m_v=x_{\sigma(v)}y_{[m]\setminus\sigma(v)}$ and
$m_w=x_{\sigma(w)}y_{[m]\setminus\sigma(w)}$ are transversal monomials of degree $m$.
Since $m_v$ and $m_w$ have the same degree and are distinct, neither divides the other, so $I_{(m)}$
is minimally generated by these two monomials.  Hence $\pd(I_{(m)})=1$, and the minimal resolution of a
two-generated monomial ideal gives
\[
\reg(I_{(m)})=\deg\lcm(m_v,m_w)-1.
\]
Moreover, $\deg\lcm(m_v,m_w)=m+d_H(v,w)=m+d$: for each coordinate where $v$ and $w$ agree,
$\lcm(m_v,m_w)$ contains exactly one of $\{x_i,y_i\}$, while for each coordinate where they differ it
contains both.  Therefore $\reg(I_{(m)})=(m+d)-1=r$.

Finally, define $\C_r:=\C^{(m)}\times \F_2^{ n-m}\subseteq\F_2^n$ and set $I_r=\P(J_{\C_r})$.
By repeated application of \Cref{lem:free-neuron}, adjoining free coordinates preserves projective dimension
and regularity. So $\pd(I_r)=1$ and $\reg(I_r)=r$.

For nondegeneracy, observe that when $m\ge 2$ deleting two words from $\F_2^m$ cannot fix any coordinate:
each coordinate takes both values on $\F_2^m$, and removing only two words leaves at least one codeword with
$0$ and at least one with $1$ in every coordinate. Hence $\C^{(m)}$ is nondegenerate, and so is its free
extension $\C_r$.
\end{proof}

\begin{remark}
The hypothesis $n\ge 3$ in \Cref{prop:pdim1-any-reg} is necessary only because the stated range
$1\le r\le 2n-1$ includes $r=2$ when $n=2$.  In our construction for $n\ge 3$ we delete two words
$v,w\in\F_2^m$ with $d_H(v,w)\ge 2$, which forces $\reg(\P(J_{\F_2^m\setminus\{v,w\}}))=m+d_H(v,w)-1$.
For $n=2$ the condition $m\le n$ forces $m=2$, and then $d_H(v,w)\ge 2$ implies $d_H(v,w)=2$. So the
construction yields $\reg(I)=3$ (deleting an antipodal pair) but cannot produce $\reg (I)=2$ where $I=\P(J_{\F_2^m\setminus\{v,w\}})$.  Attempting to
obtain $r=2$ would require deleting Hamming neighbors ($d_H(v,w)=1$), which forces a fixed coordinate and
hence a linear generator in $I$. So $\reg(I)=1$ by \Cref{thm:reg1-classification}.
Thus $(\pd(I),\reg(I))=(1,2)$ does not occur for $n=2$.
\end{remark}

\section{Small regularity: $\reg  (I) \le 3$}\label{sec:small-reg}

In this section we study polarized neural ideals in the low-regularity regime $\reg(I)\le 3$.
We treat the cases $\reg(I)=1,2,3$ in turn.  For $\reg(I)=2$ we use Hochster’s formula together with a
connectedness criterion for induced subcomplexes of the polar complex to obtain the sharp bound
$\pd(I)\le n-2$ for every nonzero proper polarized neural ideal, and we construct families realizing each value
in this range.  Along the line $\reg(I)=3$ we exhibit a family realizing every projective dimension in
$[0,2n-3]$.

\subsection{Classification of neural codes with regularity $1$}\label{sec:reg1}

Assume $\reg(I)=1$.  Then $I$ has a $1$--linear resolution, hence is generated by variables.  For polar
neural ideals, these generators record fixed coordinates: $x_i\in I$ (respectively $y_i\in I$) means that
the $i$th coordinate is identically $0$ (respectively $1$) on $\C$.  This forces $\C$ to be a coordinate
subcube of $\F_2^n$ obtained by fixing some coordinates to $0$ and some disjoint coordinates to $1$.
The next theorem makes this correspondence precise and computes $\pd(I)$ from the number of fixed
coordinates.

\begin{theorem}\label{thm:reg1-classification}
Let $\C\subset \F_2^n$ be a nonempty neural code.
Then the following are equivalent:
\begin{enumerate}
\item $\reg(I)=1$;
\item $I$ is generated by variables and, after relabeling,
\[
I=\langle x_i : i\in\sigma\rangle+\langle y_j : j\in\tau\rangle
\qquad\text{for some disjoint }\sigma,\tau\subseteq[n];
\]
\item $\C$ is a subcube of $\F_2^n$ obtained by fixing some coordinates to $0$ and some disjoint coordinates to $1$, i.e.
\[
\C=\{ v\in \F_2^n : v_i=0\ \forall i\in\sigma,\ \ v_j=1\ \forall j\in\tau \}
\]
for some disjoint $\sigma,\tau\subseteq[n]$.
\end{enumerate}
Moreover, if these conditions hold and $t:=|\sigma|+|\tau|$, then
\[
\pd(I)=t-1\le n-1,
\]
and $\pd(I)=n-1$ if and only if $t=n$, equivalently if and only if $\C$ consists of a single codeword.
\end{theorem}

\begin{proof}
\emph{$(1)\Rightarrow(2)$:}
If $\reg(I)=1$, then every minimal monomial generator of $I$ has degree $1$. Hence $I$ is generated by variables.
Since $I$ is proper, it cannot contain both $x_i$ and $y_i$ for any $i$ by \Cref{rem:basic_degree_bounds}(2).
Therefore
\[
I=\langle x_i:i\in\sigma\rangle+\langle y_j:j\in\tau\rangle
\]
for some disjoint $\sigma,\tau\subseteq[n]$.

\emph{$(2)\Rightarrow(1)$:}
Conversely, any ideal generated by variables has a $1$--linear resolution and $\reg(I)=1$.

\emph{$(2)\Rightarrow(3)$:}
Assume $I=\langle x_i:i\in\sigma\rangle+\langle y_j:j\in\tau\rangle$ with $\sigma\cap\tau=\varnothing$, and set
\[
Q:=\{ v\in\F_2^n : v_i=0\ \forall i\in\sigma,\ v_j=1\ \forall j\in\tau \}.
\]
Since $x_i\in I$, then $x_i\in J_\C$. So $x_i(c)=c_i=0$ for all $c\in\C$.
Similarly, $y_j\in I$ implies that $(1-x_j)\in J_\C$. So $(1-x_j)(c)=1-c_j=0$ for all $c\in\C$.
Thus $\C\subseteq Q$.

To prove equality, suppose $\C\subsetneq Q$ and choose $u\in Q\setminus \C$.
Then $u$ is a noncodeword, so $\rho_u\in J_\C$.
Choose a minimal pseudo-monomial $f\in J_\C$ dividing $\rho_u$. By definition
$f\in CF(J_\C)$.

We claim that $\deg(f)\ge 2$.
Indeed, since $u\in Q$, for each $i\in\sigma$ one has $u_i=0$, so $\rho_u$ is divisible by $(1-x_i)$ but
not by $x_i$, and for each $j\in\tau$ one has $u_j=1$, so $\rho_u$ is divisible by $x_j$ but not by $(1-x_j)$.
Thus none of the linear pseudo-monomials $x_i$ ($i\in\sigma$) or $(1-x_j)$ ($j\in\tau$) divides $\rho_u$.
If $k\notin\sigma\cup\tau$, then neither $x_k$ nor $(1-x_k)$ lies in $J_\C$, since otherwise the $k$th
coordinate would be fixed on $Q$, contradicting the description of $Q$.
Hence no degree-$1$ pseudo-monomial in $J_\C$ divides $\rho_u$. Therefore $\deg(f)\ge 2$.

Notice that $\P(f)\in I$.
On the other hand, none of the variable generators of $I$ in~(2) divides $\P(f)$: the same divisibility
check as above shows that neither $x_i$ ($i\in\sigma$) nor $y_j$ ($j\in\tau$) divides $\P(f)$.
This contradicts the assumption that $I$ is generated by $\{x_i:i\in\sigma\}\cup\{y_j:j\in\tau\}$.
Therefore $\C=Q$.

\emph{$(3)\Rightarrow(2)$:}
Assume $\C$ as in~(3), with $\sigma$ and $\tau$ the sets of coordinates fixed to $0$ and $1$.
Every noncodeword violates at least one fixed coordinate, so each characteristic pseudo-monomial $\rho_v$ is divisible by
some $x_i$ ($i\in\sigma$) or some $(1-x_j)$ ($j\in\tau$). Hence
\[
J_\C\subseteq \langle x_i:i\in\sigma,\ (1-x_j):j\in\tau\rangle.
\]
Conversely, for $i\in\sigma$ we have
\[
\sum_{v:  v_i=1}\rho_v
 = 
x_i\prod_{k\neq i}\bigl(x_k+(1-x_k)\bigr)
 = x_i\in J_\C,
\]
and for $j\in\tau$,
\[
\sum_{v:  v_j=0}\rho_v
 = 
(1-x_j)\prod_{k\neq j}\bigl(x_k+(1-x_k)\bigr)
 = 1-x_j\in J_\C.
\]
Thus
\[
J_\C=\langle x_i:i\in\sigma \rangle +\langle (1-x_j):j\in\tau\rangle,
\]
and polarizing gives
\[
I=\langle x_i:i\in\sigma\rangle+\langle y_j:j\in\tau\rangle.
\]

Finally, if $I$ is generated by $t=|\sigma|+|\tau|$ variables, then the Taylor on these variables
yields a  minimal free resolution of $I$ of
length $t-1$. Therefore $\pd(I)=t-1\le n-1$. Moreover, $t=n$ if and only if every coordinate is fixed,
equivalently if and only if $\C$ consists of a single codeword.
\end{proof}

\begin{remark}\label{rem:subcube-vs-cosubcube}
Coordinate subcubes control two boundary behaviors of the $\pd(I) $ and $\reg(I)$, depending on
whether the code is itself a subcube or the complement of one.  By \Cref{thm:reg1-classification},
$\reg(I)=1$ if and only if $\C$ is obtained by fixing some coordinates to $0$ and some disjoint
coordinates to $1$; equivalently, in the notation of \Cref{prop:pdim0_classification},
$\C=Q(\tau,\sigma)$ for some disjoint $\sigma,\tau\subseteq[n]$, and then $I$ is generated by variables.
By \Cref{prop:pdim0_classification}, $\pd(I)=0$ if and only if $\C=\F_2^n\setminus Q(\sigma,\tau)$ for
some nonempty disjoint $\sigma,\tau$, and then $I$ is principal.
In this sense, subcubes (and their complements) play for the $(\reg,\pd)$ boundary the same organizing role
that antipodal pairs play for the extremal corners described in
\Cref{thm:maxreg_classification,cor:antipodal_pairs_maxpdim}.

\end{remark}

\subsection{The $\reg  (I)=2$ line}\label{sec:reg2-line}

We now consider the case $\reg(I)=2$.  By Hochster’s formula, the Betti table is supported on the linear
and quadratic strands, so contributions arise only from induced subcomplexes of the polar complex that
are empty or disconnected.  The connectedness threshold in \Cref{lem:DW-connected-largeW} is the key input
that controls when such disconnectedness can occur.

\begin{proposition}\label{prop:reg2-pd-bound}
Let $\C\subseteq\F_2^n$ be a nonempty neural code. If $\reg(I)=2$, then
\[
\pd(I)\ \le\ n-2.
\]
\end{proposition}

\begin{proof}
Set $F:=\{t\in[n]: x_t\in I \text{ or } y_t\in I\}$ and $f:=|F|$.  For a polarized neural ideal we cannot have both $x_t$ and $y_t$ in $I$ by \Cref{rem:basic_degree_bounds}(2). So $f\le n$.

Let $S':=\Bbbk[\{x_i,y_i : i\in[n]\setminus F\}]$
and decompose
\[
I=L+I',\qquad 
L:=\langle x_t : x_t\in I\rangle+\langle y_t : y_t\in I\rangle,\qquad I'\subseteq S',
\]
where $I'$ is generated by the remaining minimal generators of $I$.  Since $\reg(I)=2$, we have $I'\neq 0$.

It follows from \Cref{lem:add_variable_pd_reg} that
\[
\reg_{S'}(I')=2
\qquad\text{and}\qquad
\pd_S(I)=\pd_{S'}(I')+f.
\]
Because $I'$ has no linear generators and $\reg_{S'}(I')=2$, all minimal generators of $I'$ have degree $2$
and $I'$ has a $2$--linear resolution; equivalently, the only possible nonzero Betti numbers of $I'$ lie on
the quadratic strand $j-i=2$.

Let $\Delta'$ denote the polar complex associated to $I'$ (equivalently, the polar complex of the restricted code obtained from $\C$ by deleting the coordinates in $F$; indeed, each coordinate in $F$ is constant on $\C$.).  This restricted code is nondegenerate, since $I'$
contains no variables.  By Hochster's formula, if $\beta_{a,a+2}(I')\neq 0$, then there exists
$W\subseteq V(\Delta')$ with $|W|=a+2$ such that
\[
\widetilde H_0(\Delta'_W;\Bbbk)\neq 0,
\]
i.e.\ $\Delta'_W$ is disconnected.  Applying \Cref{lem:DW-connected-largeW} to the nondegenerate code on
$n-f$ neurons shows that $\Delta'_W$ is connected whenever $|W|\ge (n-f)+1$.  Hence any $W$ witnessing
disconnection must satisfy $|W|\le n-f$. So $a\le (n-f)-2$.  It follows that
\[
\pd_S(I)=\pd_{S'}(I')+f \le ((n-f)-2)+f = n-2,
\]
which completes the proof.
\end{proof}

\begin{remark}
 \Cref{prop:reg2-pd-bound} explains what we observe along the $\reg (I)=2$ lines in \Cref{fig:plots_3_4}. In particular, this explains why  $(\pd  (I),\reg  (I))=(2,2)$ for $n=3$ and $(\pd (I),\reg (I))=(3,2)$  for  $n=4$ cannot occur.
\end{remark}

\Cref{prop:reg2-pd-bound} describes the allowable range of projective dimensions along the line $\reg=2$. We next give an explicit family of codes realizing every value $0\le \pd (I)\le n-2$ on this line.

\begin{proposition}\label{prop:reg2-line-realize-pd}
Fix $n\ge 2$. For every integer $p$ with $p \in [0, n-2]$, there exists a neural code
$\C_{n,p}\subseteq \F_2^n$ such that for $I_{n,p}:=\P(J_{\C_{n,p}})\subseteq
S$ we have
\[
(\pd(I_{n,p}),\reg(I_{n,p}))=(p,2).
\]
\end{proposition}

\begin{proof}
Fix $p\in\{0,1,\dots,n-2\}$ and define
\[
\C_{n,p}:=\{ v\in\F_2^n:\ v_1=\cdots=v_p=0\ \text{ and not }(v_{p+1}=v_{p+2}=1) \}.
\]

Recall that $J_{\C_{n,p}}$ is generated by the characteristic pseudo-monomials of $v\in \F_2^n\setminus \C_{n,p}$
\[
\rho_v  =  \prod_{i:v_i=1} x_i\prod_{j:v_j=0}(1-x_j),
\]
and $CF(J_{\C_{n,p}})$ consists of the pseudo-monomials in $J_{\C_{n,p}}$ that are minimal under divisibility.

We first show that $x_i\in J_{\C_{n,p}}$ for $i=1,\dots,p$. Indeed, every codeword in $\C_{n,p}$ satisfies $v_i=0$. So every word with $i$th coordinate $1$ is a noncodeword. Summing their characteristic pseudo-monomials gives
\[
\sum_{v:  v_i=1}\rho_v
 = 
x_i\prod_{k\neq i}\bigl(x_k+(1-x_k)\bigr)
 = x_i\in J_{\C_{n,p}}.
\]
Moreover, each $x_i$ is minimal.

Next we show $x_{p+1}x_{p+2}\in J_{\C_{n,p}}$. Since $\C_{n,p}$ contains no word with $v_{p+1}=v_{p+2}=1$,
every word with $v_{p+1}=v_{p+2}=1$ is a noncodeword, and summing their characteristic pseudo-monomials yields
\[
\sum_{v:  v_{p+1}=v_{p+2}=1}\rho_v
=
x_{p+1}x_{p+2}\prod_{k\notin\{p+1,p+2\}}\bigl(x_k+(1-x_k)\bigr)
=
x_{p+1}x_{p+2}\in J_{\C_{n,p}}.
\]
This pseudo-monomial is minimal because $\C_{n,p}$ contains words with $(v_{p+1},v_{p+2})=(1,0)$ and with $(0,1)$,
so neither $x_{p+1}$ nor $x_{p+2}$ lies in $J_{\C_{n,p}}$.

Now let $v\notin \C_{n,p}$. Then either $v_i=1$ for some $i\le p$, in which case $\rho_v$ is divisible by $x_i$,
or else $v_1=\cdots=v_p=0$ and $v_{p+1}=v_{p+2}=1$, in which case $\rho_v$ is divisible by $x_{p+1}x_{p+2}$.
Thus every $\rho_v$ is divisible by one of $x_1,\dots,x_p,x_{p+1}x_{p+2}$. So these pseudo-monomials generate $J_{\C_{n,p}}$, and by the
minimality statements above they are exactly the canonical form. Thus
\[
I_{n,p}=\P(J_{\C_{n,p}})=\bigl\langle x_1,\dots,x_p,\ x_{p+1}x_{p+2}\bigr\rangle.
\]

Let $G(I_{n,p})=\{x_1,\dots,x_p,u\}$ where $u:=x_{p+1}x_{p+2}$. For any subset $A\subseteq G(I_{n,p})$,
\[
\lcm(A)=
\begin{cases}
\prod_{x_i\in A} x_i , & u\notin A,\\[4pt]
x_{p+1}x_{p+2}\prod_{x_i\in A} x_i, & u\in A.
\end{cases}
\]
In particular, distinct subsets yield distinct least common multiples, so the Taylor resolution of $I_{n,p}$
is minimal and has length $p$. Therefore
$\pd(I_{n,p})=p$.

To compute regularity, for a subset $A$ of size $s$ the corresponding Taylor basis element lies in homological
degree $s-1$. If $u\notin A$, then $\deg(\lcm(A))=s$; if $u\in A$, then
$\deg(\lcm(A))=s+1$. Hence $\reg(I_{n,p})=2$.
\end{proof}

\subsection{The $\reg  (I)=3$ line}\label{sec:reg3}

In this subsection we work along the line $\reg(I)=3$ by building a flexible family of examples.
We start with the all-or-nothing code on $m$ coordinates; by \Cref{thm:all-or-nothing} its polar ideal
satisfies $\reg(I)=3$ and $\pd(I)=2m-3$.  We then fix $t$ additional coordinates, which adjoins $t$ linear
generators in variables disjoint from the all-or-nothing part.  This operation preserves regularity while
adjusting projective dimension in a controlled way.

\begin{construction}\label{const:reg3_tunable}
Fix $n\ge 3$. Choose a partition
\[
[n]=A\sqcup B\sqcup F
\qquad\text{with}\qquad
|A|=m,\ |B|=t,\ |F|=n-m-t,
\]
where $2\le m\le n$ and $0\le t\le n-m$.
Further split $B=B_0\sqcup B_1$ (coordinates fixed to $0$ and to $1$).

Define a code $C\subseteq\F_2^n$ by requiring:
\begin{itemize}
\item on $A$, only the two patterns $0^A$ and $1^A$ occur;
\item on $B_0$ we impose $v_i=0$ and on $B_1$ we impose $v_j=1$;
\item on $F$ all bits are free.
\end{itemize}
Equivalently,
\[
\C  = 
\Big(\{0^A,1^A\}\times\{0\}^{B_0}\times\{1\}^{B_1}\times \F_2^{F}\Big)
\subseteq \F_2^n.
\]
\end{construction}

\begin{remark}
 When $t>0$ this construction has fixed coordinates (hence the code is degenerate in our
sense), but this is harmless for the homological invariants we compute below.   
\end{remark}

\begin{proposition}\label{prop:reg3_tunable_invariants}
Let $\C$ be as in Construction~\ref{const:reg3_tunable}.
Then
\[
I  = 
\big\langle x_i y_j : i,j\in A,\ i\neq j \big\rangle
 + 
\langle x_i : i\in B_0\rangle
 + 
\langle y_j : j\in B_1\rangle,
\]
and
\[
(\pd(I),\reg(I))= ((2m-3)+t,3).
\]
\end{proposition}

\begin{proof}
Write $S=S_A\otimes_k S_B\otimes_k S_F$, where
\[
S_A=k[x_i,y_i:i\in A],\qquad
S_B=k[x_i,y_i:i\in B],\qquad
S_F=k[x_i,y_i:i\in F].
\]
Set
\[
I_A:=\big\langle x_i y_j : i,j\in A,\ i\neq j\big\rangle \subseteq S_A,
\qquad
L:=\langle x_i : i\in B_0\rangle+\langle y_j : j\in B_1\rangle \subseteq S_B.
\]
The variables appearing in $I_A$ and in $L$ are disjoint, and $L$ is generated by $t$ variables.

We first identify the generators of $I$. Since $\C$ restricts on $A$ to the two words $0^A$ and $1^A$,
no codeword has $v_i=1$ and $v_j=0$ for distinct $i,j\in A$. Thus $x_i(1-x_j)\in J_\C$; it is minimal
because both values occur in each coordinate of $A$. Polarizing gives $x_i y_j\in I$ for all $i\neq j$ in $A$.
On $B_0$ every codeword has $v_i=0$, so $x_i\in J_\C$ minimally and hence $x_i\in I$ for $i\in B_0$.
Similarly, on $B_1$ every codeword has $v_j=1$, so $(1-x_j)\in J_\C$ minimally and polarizes to $y_j\in I$.

Conversely, any noncodeword either violates a fixed coordinate in $B$ (forcing divisibility by a generator of $L$)
or differs from both $0^A$ and $1^A$ on $A$, in which case it contains a $1/0$ mismatch among coordinates of $A$
and is divisible by some $x_i(1-x_j)$ with $i\neq j$. Hence these are exactly the minimal generators, and
\[
I = I_A\cdot S + L\cdot S.
\]

Since $I_A$ is the all-or-nothing ideal on $m=|A|$ coordinates, \Cref{thm:all-or-nothing} gives
\[
(\pd(I_A), \reg(I_A))=(2m-3,3).
\]
Adjoining the $t$ generators of $L$ one at a time, \Cref{lem:add_variable_pd_reg} increases projective dimension
by $1$ at each step and leaves regularity unchanged. Therefore
\[
(\pd(I), \reg(I))=((2m-3)+t,3).\qedhere
\]
\end{proof}

Proposition~\ref{prop:reg3_tunable_invariants} reduces the problem of realizing projective dimensions on the line
$\reg=3$ to choosing parameters $(m,t)$ with $(2m-3)+t=p$. We now show that this is possible for every
$0\le p\le 2n-3$.

\begin{theorem}\label{thm:reg3_all_pdim}

Fix $n\ge 3$. For every integer $p\in[0,2n-3]$, there exists a neural code $\C\subseteq\F_2^n$
such that 
\[
(\pd(I), \reg(I))=(p,3).
\]
\end{theorem}

\begin{proof}
If $p=0$, choose disjoint $\sigma,\tau\subseteq[n]$ with $|\sigma|+|\tau|=3$ and set
$C=\F_2^n\setminus Q(\sigma,\tau)$. Then $I=\langle x_\sigma y_\tau\rangle$ is principal of degree $3$. So $\pd(I)=0$ and $\reg(I)=3$.

Now assume $1\le p\le 2n-3$.
If $1\le p\le n-1$, apply Construction~\ref{const:reg3_tunable} with $m=2$ and $t=p-1$.
Then Proposition~\ref{prop:reg3_tunable_invariants} gives
\[
\pd(I)=(2\cdot 2-3)+(p-1)=p
\qquad\text{and}\qquad
\reg(I)=3.
\]

If $n\le p\le 2n-3$, set
\[
m:=p-n+3 \quad\text{and}\quad t:=2n-3-p.
\]
Then $3\le m\le n$ and $0\le t\le n-m$, so Construction~\ref{const:reg3_tunable} applies.
Proposition~\ref{prop:reg3_tunable_invariants} yields
\[
\pd(I)=(2m-3)+t
=\bigl(2(p-n+3)-3\bigr)+(2n-3-p)=p,
\qquad
\reg(I)=3. \qedhere
\]
\end{proof}

\section{Constructions of neural codes below the line $p+r\le n$}\label{sec:last}

In this section we construct nondegenerate neural codes whose polarized neural ideals realize every pair
$(p,r)=(\pd(I),\reg(I))$ in the region $p+r\le n$, with $0\le p\le n-1$ and $1\le r\le n$.
Our main input is that simplicial codes have neural ideals generated purely by Type~1 relations, so the
polar ideal is a Stanley--Reisner ideal \cite[Lemma~4.4]{smb}.  Indeed,  any squarefree monomial ideal only in the $x$-variables are polarized neural ideals generated by Type~1 relations. By focusing on simplicial complexes, we use
Auslander--Buchsbaum and Cohen--Macaulayness to control projective dimension and Hochster’s formula to detect
regularity via top homology.  Finally, adding free neurons extends constructions from $m=p+r$ to arbitrary
$n$ without changing $\pd$ or $\reg$.

\begin{lemma}\label{lem:simplicial-code-SR}
Let $\Delta$ be a simplicial complex on $[m]$ and set
\[
\C_\Delta:=\{c\in\F_2^m:\supp(c)\in\Delta\}.
\]
Then the neural ideal is the Stanley--Reisner ideal,
\[
J_{\C_\Delta}=I_\Delta\subseteq k[x_1,\dots,x_m],
\]
and hence
\[
\P(J_{\C_\Delta})  =  I_\Delta\cdot T,
\qquad
T:=k[x_1,\dots,x_m,y_1,\dots,y_m].
\]
Moreover, the induced subcomplex of the polar complex $\Delta_{\C_\Delta}$ on the $x$--vertices
$\{x_1,\dots,x_m\}$ is $\Delta$.
\end{lemma}

\begin{proof}
By construction, $\supp(\C_\Delta)=\Delta$. Since $\Delta$ is a simplicial complex, it was shown in \cite{smb}
 that the canonical form $CF(J_{\C_\Delta})$ has no Type~2 or Type~3 relations and that $J_{\C_\Delta}$
is generated by the Type~1 relations, i.e.\ by the squarefree monomials corresponding to minimal nonfaces of
$\Delta$ \cite[Lemma~4.4]{smb}. Thus $J_{\C_\Delta}=I_\Delta$ in $k[x_1,\dots,x_m]$, and polarization yields
$\P(J_{\C_\Delta})=I_\Delta\cdot T$.

Finally, since $\P(J_{\C_\Delta})$ is generated by monomials in the $x$--variables, the induced subcomplex of
$\Delta_{\C_\Delta}$ on $\{x_1,\dots,x_m\}$ has faces exactly those $F\subseteq[m]$ for which $x_F\notin I_\Delta$,
namely $F\in\Delta$. Hence $(\Delta_{\C_\Delta})|_{\{x_1,\dots,x_m\}}=\Delta$.
\end{proof}

We now realize all pairs $(\pd (I),\reg(I))$ in the region $p+r\le n$.  Since the cases of $\reg=1,2,3$
were handled in \Cref{sec:reg1,sec:reg2-line,sec:reg3},  throughout we assume $r\ge 4$.

\begin{theorem}\label{thm:pr-band-realization}
Fix integers $n\geq 4$, $p\in [0, n-1]$, and $r\in [4,n]$ with $p+r\le n$.
Then there exists a nondegenerate neural code $\C\subseteq \F_2^n$ such that
\[
(\pd(I),\reg(I))=(p,r).
\]
\end{theorem}

\begin{proof}
Put $m:=p+r$, so $m\le n$ and $m\ge r$. Choose an $(r-2)$--dimensional simplicial sphere $\Delta$ on
vertex set $[m]$; for instance, one may take a stacked $(r-2)$--sphere with $m$ vertices.

Define the simplicial code
\[
\C'=\C_\Delta=\{c\in \F_2^m : \supp(c)\in \Delta\}.
\]
Since $\emptyset\in\Delta$, we have $0\cdots 0\in \C'$. Moreover, because $\Delta$ has vertex set $[m]$
(no ghost vertices), each singleton $\{i\}$ lies in $\Delta$, so the word with support $\{i\}$ lies in $\C'$.
Hence for every $i\in[m]$ there are codewords in $\C'$ with $i$th coordinate equal to $0$ and to $1$, and
$\C'$ is nondegenerate. Now add $n-m$ free coordinates and set
\[
\C  :=  \C'\times \F_2^{ n-m} \subseteq \F_2^n.
\]
This preserves nondegeneracy, since each of the last $n-m$ coordinates attains both values in~$\C$.

Let $S_m:=k[x_1,\dots,x_m]$, let $T:=k[x_1,\dots,x_m,y_1,\dots,y_m]$. By Lemma~\ref{lem:simplicial-code-SR},
\[
\P(J_{\C'})=I_\Delta\cdot T.
\]
Iterating \Cref{lem:free-neuron} to account for the $n-m$ free coordinates yields
\[
I=\P(J_{\C'})\cdot S = I_\Delta\cdot S.
\]
In particular, adjoining these extra variables does not change projective dimension or regularity, so it
suffices to compute $\pd(I_\Delta)$ and $\reg(I_\Delta)$ over $S_m$.

Because $\Delta$ is a simplicial sphere, the Stanley--Reisner ring $S_m/I_\Delta$ is Cohen--Macaulay
(by Reisner's criterion; see, e.g., \cite[Theorem~5.1.9]{miller}), and
\[
\dim(S_m/I_\Delta)=\dim\Delta+1=r-1.
\]
Hence $\depth(S_m/I_\Delta)=r-1$, and Auslander--Buchsbaum gives
\[
\pd_{S_m}(I_\Delta)+1= \pd_{S_m}(S_m/I_\Delta)=m-\depth(S_m/I_\Delta)=m-(r-1)=p+1.
\]
% From the short exact sequence $0\to I_\Delta\to S_m\to S_m/I_\Delta\to 0$, 
It follows that $\pd_{S_m}(I_\Delta)=p$.

For regularity, apply Hochster’s formula to $I_\Delta$. Since $\Delta$ is an $(r-2)$--sphere on $m$ vertices,
$\widetilde H_{r-2}(\Delta;k)\cong k$. Taking $W=V(\Delta)$ (so $|W|=m$ and $\Delta_W=\Delta$) yields
\[
\beta_{p,m}(I_\Delta)
=\dim_k \widetilde H_{m-p-2}(\Delta;k)
=\dim_k \widetilde H_{r-2}(\Delta;k)
=1,
\]
and therefore $\reg(I_\Delta)\ge m-p=r$.
Conversely, if $\beta_{i,j}(I_\Delta)\neq 0$, then Hochster’s formula gives
$\widetilde H_{j-i-2}(\Delta_W;k)\neq 0$ for some $W$ with $|W|=j$. Hence
\[
j-i-2 \le \dim(\Delta_W)\le \dim(\Delta)=r-2,
\]
so $j-i\le r$ for every nonzero graded Betti number. Thus $\reg(I_\Delta)\le r$, and we conclude that
$\reg(I_\Delta)=r$.

Finally, since $I=I_\Delta\cdot S$ is obtained from $I_\Delta$ by adjoining variables, \Cref{lem:free-neuron}
implies that $\pd$ and $\reg$ are preserved. Therefore $\pd(I)=p$ and $\reg(I)=r$.
\end{proof}

\begin{ex}\label{ex:n8p3r4}
Let $n=8$, $p=3$, $r=4$. Here $m=p+r=7$ and $r-2=2$, so we need a $2$--sphere $\Delta$ on vertices $[7]$. We construct $\Delta$ as a stacked $2$-sphere on $[7]$ by starting from the tetrahedron boundary
$\partial\Delta_3$ on $\{1,2,3,4\}$ and successively stacking on the facets $123$, $124$, and $134$,
introducing the new vertices $5,6,7$ in turn. Then its  facets (triangles) are 
\[
\begin{aligned}
\mathcal{F}(\Delta)=\{&
234, 125, 235, 135, 126, 246, 146, 137, 347, 147\}.
\end{aligned}
\]
A direct check shows that the minimal nonfaces of $\Delta$ are
\[
27, 36, 45, 56, 57, 67, 123, 124, 134.
\]
Define $\C'=\C_\Delta\subseteq \F_2^7$ by $\supp(c)\in\Delta$, and then set 
\[
\C:=\C'\times \F_2  \subseteq \F_2^8.
\]

By Lemma~\ref{lem:simplicial-code-SR}, the polarized neural ideal is 
generated by these minimal nonfaces (now viewed in $k[x_1,\dots,x_8,y_1,\dots,y_8]$):
\[
\P(J_\C)
=
\langle 
x_2x_7, x_3x_6, x_4x_5, x_5x_6, x_5x_7, x_6x_7, 
x_1x_2x_3, x_1x_2x_4, x_1x_3x_4
\rangle.
\]
By Theorem~\ref{thm:pr-band-realization}, this ideal satisfies
\[
\pd(\P(J_\C))=3
\qquad\text{and}\qquad
\reg(\P(J_\C))=4.
\]

Equivalently (as a description of the code), $\C$ consists of all words in $\F_2^8$ whose support
on the first $7$ coordinates does \emph{not} contain any of the forbidden sets
\[
\{2,7\},\{3,6\},\{4,5\},\{5,6\},\{5,7\},\{6,7\},\{1,2,3\},\{1,2,4\},\{1,3,4\},
\]
with the $8$th coordinate arbitrary.
\end{ex}

\section{Conclusion and Further directions}\label{sec:further-directions}

The results above suggest an emerging dictionary between elementary features of the Hamming cube and
the homological data of the polarized neural ideal $\P(J_\C)$. Throughout, we translate the
combinatorics of $\C\subseteq\F_2^n$ into the topology of induced subcomplexes of the polar complex
$\Delta_\C$ via Hochster’s formula. Several sharp boundary phenomena arise from particularly rigid
cube structure: coordinate subcubes and complements of coordinate subcubes control 
$\reg(\P(J_\C))=1$ and $\pd(\P(J_\C))=0$, while antipodality controls the extremal corners
$\reg(\P(J_\C))=2n-1$ and $\pd(\P(J_\C))=2n-3$. These patterns raise the prospect that the achievable
region of $(\pd,\reg)$ values,  even finer invariants such as the shape of the Betti table, may
be describable using a small amount of combinatorial data, and that additional geometric restrictions
on $\C$ (e.g.\ convexity or hyperplane structure) may dramatically constrain what can occur. We
conclude our study by collecting several questions motivated by this viewpoint.

\begin{question}\label{ques:duality-principle}
To what extent is the pair $(\pd(\P(J_\C)),\reg(\P(J_\C)))$ controlled by Hamming-cube geometry of
$\C\subseteq \F_2^n$?  Are the boundary classifications (subcubes/co-subcubes and antipodal configurations)
shadows of a broader cube-theoretic organization of the entire achievable $(\pd,\reg)$ region, or do
genuinely new behaviors appear away from the boundary?
\end{question}

\begin{question}\label{ques:region}
For fixed $n$, determine the set \[ \mathcal{R}_n:=\bigl\{(\pd(\P(J_\C)),\reg(\P(J_\C))) : \C\subseteq \F_2^n \bigr\}. \] 
\end{question}

\begin{question}\label{ques:betti}
Give a combinatorial description of the multigraded Betti numbers of $\P(J_\C)$ in terms of
Hamming-cube structure of $\C$ (or of its complement $\F_2^n\setminus\C$).
\end{question}

\begin{question}\label{ques:restricted-classes}
For a class $\mathfrak{F}$ of codes (e.g.\ convex codes, stable hyperplane codes), determine
\[
\mathcal{R}_n(\mathfrak{F}) := \{(\pd(\P(J_\C)),\reg(\P(J_\C))) : \C\in\mathfrak{F},\ \C\subseteq\F_2^n\}.
\]
How does $\mathcal{R}_n(\mathfrak{F})$ compare to $\mathcal{R}_n$?
\end{question}

\textbf{Acknowledgements.} We thank Rebecca R. G. and Hugh Geller for inspiring this project and for many helpful conversations. Portions of this work grew out of the second author E.~Lew’s master’s thesis, completed at Bryn Mawr College in May 2025, and were further developed during first author's visit to the Fields Institute for the Thematic Program in Commutative Algebra and Applications in June 2025. The authors thank both Bryn Mawr College and the Fields Institute for their support. The first author (Kara) was supported by NSF grant DMS-2418805.

\bibliographystyle{abbrv}
\bibliography{refs}

\end{document}